%% file: AltMU23.tex
\documentclass[11pt,reqno]{amsart}
\pdfoutput=1

\usepackage[english]{babel}
\usepackage[utf8]{inputenc}
\usepackage[T1]{fontenc}
\usepackage{lmodern} \normalfont 
\DeclareFontShape{T1}{lmr}{bx}{sc} { <-> ssub * cmr/bx/sc }{}
\usepackage{microtype}	

\usepackage{amssymb}
\usepackage{amsmath}
\usepackage{amsthm}
\usepackage{mathtools}
\usepackage{mathrsfs}
\usepackage{accents} 
\usepackage{etoolbox}
\usepackage{siunitx}
\usepackage{dsfont} 
\usepackage{graphicx}
\usepackage{color}
\usepackage[dvipsnames]{xcolor}
\usepackage{circuitikz}
\usepackage{tikz}
\usetikzlibrary{calc,positioning,shapes}
\usetikzlibrary{patterns,decorations.pathmorphing,decorations.markings}
\usepackage{pgfplots}
\pgfplotsset{compat=newest}
\usepackage[margin=10pt,font=small,labelfont=bf,labelsep=endash]{caption}
\usepackage{subcaption}
\usepackage{multicol}

\usepackage{booktabs}
\usepackage{paralist}
\usepackage{soul}

\numberwithin{equation}{section}

\usepackage{algorithm}
\usepackage{algorithmicx}
\usepackage{algpseudocode}

\usepackage{cite}
\usepackage{multirow} 
\usepackage{enumitem} 
\setlist[enumerate]{label=(\roman*)}
\usepackage{xspace}
\usepackage{listings}

\usepackage[colorlinks,linkcolor=teal,citecolor=teal,urlcolor=teal]{hyperref}
\usepackage[nameinlink,noabbrev]{cleveref}

\theoremstyle{plain}
\newtheorem{theorem}{Theorem}[section]
\newtheorem{proposition}[theorem]{Proposition}
\newtheorem{lemma}[theorem]{Lemma}

\newtheorem{remark}[theorem]{Remark}

\newtheorem{example}[theorem]{Example}

\input{def}

\newcommand{\bdf}[2]{\Xi_{#1}{(#2)}}

\newcommand{\BDF}{\textsf{BDF}\xspace}

\title{Higher-order iterative decoupling for poroelasticity} 
\author{Robert Altmann${}^\dagger$ \and Abdullah Mujahid${}^{\star}$ \and Benjamin Unger${}^\star$}

\address{${}^{\dagger}$ Institute of Analysis and Numerics, Otto von Guericke University Magdeburg, Universit\"atsplatz 2, 39106 Magdeburg, Germany}
\email{robert.altmann@ovgu.de}
\address{${}^{\star}$ Stuttgart Center for Simulation Science (SC SimTech), University of Stuttgart, Universit\"{a}tsstr.~32, 70569 Stuttgart, Germany}
\email{\{abdullah.mujahid,benjamin.unger\}@simtech.uni-stuttgart.de}
\date{\today}

\begin{document}

\begin{abstract}
For the iterative decoupling of elliptic--parabolic problems such as poroelasticity, we introduce time discretization schemes up to order $5$ based on the backward differentiation formulae. Its analysis combines techniques known from fixed-point iterations with the convergence analysis of the temporal discretization. As main result, we show that the convergence depends on the interplay between the time step size and the parameters for the contraction of the iterative scheme. Moreover, this connection is quantified explicitly, which allows for balancing the single error components. Several numerical experiments illustrate and validate the theoretical results, including a three-dimensional example from biomechanics.
\end{abstract}

\maketitle

%
%

%
%
\section{Introduction}
\label{sec:intro}

The poroelasticity equations are a set of coupled partial differential equations (PDEs) describing the flow of fluids in porous materials. Initially, this model was introduced in~\cite{Bio41} for the consolidation of soil. The set of coupled PDEs in the quasi-static case consists of an elliptic equation for the mechanics and a parabolic equation for the fluid flow. The overall system is thus called an \emph{elliptic--parabolic PDE}. Among others, poroelasticity finds application in geomechanics~\cite{Zob10,KosSS80,WhiB08} and the macroscopic modeling of organs such as the human brain~\cite{EliRT23,CorADQ23} or the human heart~\cite{BarGDZVQ22}.

The well-posedness of such an elliptic--parabolic PDE is studied in~\cite{Sho00}. With an implicit Euler discretization in time, the spatial adaptivity is considered in~\cite{ErnM09}.
Spatial discretizations of this coupled problem result in block matrices, which have symmetric positive definite blocks on the diagonal.
With fully coupled methods, the computational burden is tremendous in three-dimensional applications, especially when employing higher-order discretization methods.
Hence, decoupling methods attempt to separate the flow and the mechanic equations -- either iteratively or non-iteratively -- such that efficient preconditioners for the subsystems can be used~\cite{LeeMW17}. 

For \emph{non-iterative methods}, first proposed and analyzed in~\cite{AltMU21b}, the convergence of the decoupled time-iteration is guaranteed whenever a certain \emph{weak coupling condition} is satisfied.
Such conditions are described by restrictions on the parameters of the PDE that quantify the strength of the coupling.
This framework also allows for a generalization to higher-order methods (in time)~\cite{AltMU24}.
To extend the stability regions, i.e., to relax the weak coupling condition, a procedure with an a priori known number of additional inner iterations for each time step is introduced in~\cite{AltD24,AltD24ppt}.
Extensions for corresponding nonlinear problems are considered in~\cite{AltMU21a,AltM22}.

\emph{Iterative methods} decouple the system by either first solving the mechanics equation and then the flow equation or vice versa with additional stabilization terms in a fixed-point iteration framework. The operator splitting is motivated by physics and is known as \emph{drained} and \emph{undrained splits}~\cite{KimTJ11a} as well as \emph{fixed stress} and \emph{fixed strain splits}~\cite{KimTJ11b, WheG07}.
The corresponding analysis shows the existence of a fixed-point either in the fully continuous case or the fully discrete case.
In the fully continuous case, the iterative solution is proven to converge to the solution of the fully coupled problem, hence accounting for the errors due to operator splitting.
As an example, the fixed-point analysis for the fully continuous case, resulting in the popular fixed-stress and undrained split methods, is shown in~\cite{MikW13}. 
In the fully discrete case, which considers either a first-order in time discretization~\cite{StoBKNR19} or the generalized midpoint rule~\cite{KimTJ11a,KimTJ11b}, the iterative solution is shown to reach the fixed-point, which can be shown to be a solution when applying the same discretization to the coupled problem.
Hence, one accounts again for the errors due to splitting alone of the discrete operators~\cite{StoBKNR19}.

In this article, we focus on the interaction between time integration and
operator splitting.
For this, we consider the fixed-stress operator splitting to decouple the given elliptic--parabolic problem and employ higher-order time discretizations, using the \emph{backward differentiation formulae} of $k$-th order (denoted as \BDF-$k$) for $k \le 5$. 
The choice of BDF methods over Runge--Kutta methods is more natural. 
This is due to the fact that higher-order Runge--Kutta schemes include the tensor product of the Butcher tableau with the decoupled block matrices resulting in larger
linear systems. Moreover, symmetry and positive definiteness of the decoupled block matrices is destroyed and transform based methods would be needed to prove full convergence~\cite{LubO95}. 

The results of this article are twofold. First, we present a rigorous
convergence analysis, which consists of convergence in time for the higher-order
time discretization combined with the fixed-point analysis of the operator splitting.
In particular, this accounts for the interaction between the two error components.
As a direct consequence of the presented analysis, we propose strategies for
optimizing the required number of inner iterations to reach a given tolerance
based on the information on the order of the time discretization and the step size.
Quantitatively, the relative tolerance $\tol$ for the stopping criteria is proposed
as $\tol \approx \tau^{k+3/2}$ for the \BDF-$k$ method, where $\tau$ is the time step size.
Second, the operator splitting is revisited with an academic toy example to study
the optimal choice of the stabilization term. 

The organization of this article is as follows.
In \Cref{sec:prelim}, the problem setting is introduced and we recall stability
results of \BDF methods applied to PDEs.
The \BDF-$k$ time discretization of the iteratively decoupled problem is given
in \Cref{sec:iterativeBDF} and the main result, which is the convergence of this
higher-order iterative solution to the true solution, is presented
in \Cref{th:convergence}.
Moreover, we analyze and quantify the interplay between the time-step size and
the iterative parameters that control the iterative contraction.
Finally, the expected convergence rates are demonstrated numerically
in \Cref{sec:numerics:orders}.
Furthermore, we study the balancing of the iterative parameters and the time step
size in \Cref{sec:numerics:balancing} and present a real-world example from
biomechanics in \Cref{sec:numerics:brain}. 

\subsection*{Notation}  
Throughout the paper, we write~$a \lesssim b$ to indicate that there exists a
generic constant~$C > 0$, independent of spatial and temporal discretization
parameters, such that $a \leq C b$. 
Further, the space of symmetric positive definite matrices of size $k$ is denoted
by~$\spd{k}$. 
The space of $k$-times continuously differentiable functions on the time
interval~$[0,T]$ with values in a Banach space~$\calX$ is denoted as
$\cC^{k}([0,T],\calX)$.
%
%
\section{Preliminaries and Problem Formulation}
\label{sec:prelim}

Let~$\Omega \subseteq \R^m$, $m\in\{2,3\}$ be a bounded Lipschitz domain and
$[0,T]$ the time interval where $0<T<\infty$.
We consider the quasi-static Biot poroelasticity model introduced in~\cite{Bio41}. 
Therein, the task is to find the deformation
~$u\colon [0,T]\times\Omega\rightarrow\R^{m}$ and the pressure
variable~$p\colon [0,T]\times\Omega\rightarrow\R$, satisfying the system of coupled
PDEs,
\begin{subequations}
	\label{eq:pdes}
	\begin{align}
		- \nabla\cdot\sigma(u) + \alpha \nabla p 
		&= {\hat f} \qquad\text{in }(0,T]\times\Omega, \label{eq:pdes:a}\\
		\partial_{t} \big(\alpha \nabla\cdot u + \tfrac{1}{M} p\big) - \nabla\cdot(\kappa\nabla p)
		&= {\hat g} \qquad\text{in } (0,T]\times\Omega. \label{eq:pdes:b}
	\end{align}
\end{subequations}
Here, $\sigma$ denotes the linear stress tensor
\begin{equation*}
	\sigma(u) 
	= \mu\, \big(\nabla u + (\nabla u)^\T \big) 
	+ \lambda\, (\nabla \cdot u) \id
\end{equation*}
with Lam\'{e} coefficients ${\lambda}$ and ${\mu}$. The quantity $\kappa$ denotes the permeability,$\alpha$ denotes the Biot--Willis fluid--solid coupling coefficient,
and $M$ the Biot moduli.
All these quantities are positive constants. 
Moreover, system~\eqref{eq:pdes} is equipped with appropriate initial data and
boundary conditions. 
This model is based on a quasi-static assumption where the inertial term is
neglected in \eqref{eq:pdes:a} and, hence, results in a set of coupled
elliptic--parabolic PDEs.
The model can also be characterized as {\em partial differential--algebraic
equation} (PDAE), since a spatial discretization yields a differential--algebraic
equation, cf.~\cite{AltMU21c}. 
\begin{example}[multiple network case]
\label{exp:network}
	A generalization of~\eqref{eq:pdes} is the so-called {\em multiple network
	poroelasticity theory}~\cite{BarZK60}, which finds application in the 
	modelling of the human brain~\cite{VarCT+16,EliRT23}. Instead of a single
	pressure variable, a set of $J \in \N$ pressure variables~$p_{i}\colon [0,T]\times\Omega\rightarrow\R$, $i = 1, \dots, J$ are sought,
	satisfying 
	\begin{align*}
		- \nabla\cdot\sigma(u) + \sum_{i=1}^{J}\alpha_i \nabla p_i 
		&= {\hat f} \qquad\text{in } (0,T]\times\Omega, \\
		\partial_{t} \Big(\alpha_i \nabla\cdot u + \tfrac{1}{M_i} p_{i} \Big) - \nabla\cdot \big(\kappa_{i}\nabla p_i \big) + \sum_{j\ne i}^{J}\beta_{ij}(p_i - p_j)
		&= {\hat g}_{i} \qquad\text{in } (0,T]\times\Omega.
	\end{align*}
	Here, the parameters $\alpha_i$, $M_i$, and $\kappa_i$ denote the quantities
	for the corresponding fluid network and $\beta_{ij} > 0$ represent
	so-called exchange rates between the fluid networks.  
\end{example}
The single fluid network \eqref{eq:pdes} is a special case of the multiple network
case from Example~\ref{exp:network} for $J=1$.
As illustrated in~\cite{AltMU21c}, however, the multiple network model has the same
structure as the single network case.
Hence, we restrict the following analysis to the single fluid network case, and only
mention relevant remarks about the multiple network case in the results.
As a first step, we introduce an abstract formulation of the model as an
elliptic--parabolic problem. 

\subsection{Abstract formulation as elliptic--parabolic problem}
\label{sec:prelim:parabolic}

Consider the Hilbert spaces
\[
  \cV\vcentcolon=[H_{0}^{1}(\Omega)]^m, \qquad
  \cQ\vcentcolon=H_{0}^{1}(\Omega), \qquad
  \cHV\vcentcolon=[L^{2}(\Omega)]^m, \qquad 
  \cHQ\vcentcolon=L^{2}(\Omega),
\]
resulting in the Gelfand triples $\cV \hook \cHV\simeq\cHV^{*}\hook\cV^{*}$ and
$\cQ \hook \cHQ\simeq\cHQ^{*}\hook\cQ^{*}$, cf.~\cite[Sect.~23.4]{Zei90}.
Moreover, we introduce the bilinear forms~$a\colon \cV\times\cV \to \R$,
$b\colon \cQ\times\cQ\to\R$, and $c\colon \cHQ\times\cHQ\to\R$, which we assume to
be symmetric, continuous, and elliptic in their respective spaces. 
The corresponding ellipticity and continuity constants of the bilinear forms
$\mathfrak{a}\in\{a, b, c\}$ are denoted by $c_{\mathfrak{a}}, C_{\mathfrak{a}} > 0$,
respectively.
Furthermore, let~$d\colon\cV\times\cHQ\to\R$ be a bounded bilinear form, i.e., we
assume that there exists a positive constant~$C_d > 0$ such
that~$d(u,p) \leq C_d \Vert u \Vert_{\cV} \Vert p\Vert_{\cHQ}$ 
for all~$u \in \cV$, $p\in \cHQ$.
Based on these quantities, we define the parameter 
\[
	\omega\vcentcolon=\frac{C_d^{2}}{c_a c_c},
\]
which quantifies the strength of the coupling between the elliptic and the
parabolic equations. 

Given sufficiently smooth source terms $f\colon[0,T] \to \cV^*$ and
$g\colon[0,T] \to \cQ^*$, we now seek abstract functions
$u\colon (0,T]\rightarrow\cV$ and $p\colon (0,T]\rightarrow\cQ$ such that
for a.\,e.~$t \in [0,T]$ it holds that
\begin{subequations}
\label{eq:ellpar}
\begin{align}
	a(u,v) - d(v, p) 
	&= \langle f, v \rangle, \label{eq:ellpar:a} \\
	d(\dot u, q) + c(\dot p,q) + b(p,q) 
	&= \langle g, q\rangle \label{eq:ellpar:b} 
\end{align}
for all test functions~$v\in \cV$, $q \in \cQ$. 
Furthermore, we assume {\em consistent} initial data
\begin{align}
	u(0) = u^0 \in \cV, \qquad 
	p(0) = p^0 \in \cHQ, \label{eq:ellpar:c} 
\end{align}
in the sense that we assume $a(u^0,v) - d(v, p^0) = \langle f(0),v\rangle$
for all~$v\in \cV$.
\end{subequations}
Without the coupling term represented by the bilinear form $d$,
equation~\eqref{eq:ellpar:a} is elliptic in $u$ and equation~\eqref{eq:ellpar:b} is
parabolic in $p$.
This explains why system~\eqref{eq:ellpar} is called an elliptic--parabolic problem. 

The bilinear forms for \eqref{eq:pdes} are defined as
\begin{align*}
	a(u,v) &\vcentcolon= \int_\Omega \sigma(u) : \varepsilon(v) \dx, \;\; 
	&&b(p,q) \vcentcolon= \int_\Omega \frac{\kappa}{\nu}\, \nabla p \cdot \nabla q \dx,\\
	c(p,q) &\vcentcolon= \int_\Omega \frac{1}{M}\, p\, q \dx, \;\; 
	&&d(u,q) \vcentcolon= \int_\Omega \alpha\, (\nabla \cdot u)\, q \dx.
\end{align*}
To see that the assumed properties are satisfied, we refer to~\cite{ErnM09,AltMU21b}.

The weak formulation~\eqref{eq:ellpar} can also be written in terms of operators in
the respective dual spaces.
For this, let $\cA\colon\cV\rightarrow\cV^*$, $\cB\colon\cQ\rightarrow\cQ^*$,
$\cC\colon\cHQ\rightarrow\cHQdual$, and $\cD\colon\cV\rightarrow\cHQ^{*}$ be the
operators associated with the bilinear forms $a$, $b$, $c$, and $d$ respectively. 
This then leads to the equivalent formulation 
\begin{subequations}
\label{eq:ellpar:opt}
    \begin{align}
        \cA u - \cD^{*} p 
        &= f \qquad \text{in } \cV^{*}, \label{eq:ellpar:opt:a}\\
        \cD{\dot u} + \cC{\dot p} + \cB p 
        &= g \hspace{0.79cm} \text{in } \cQ^{*}.\label{eq:ellpar:opt:b}
    \end{align}
\end{subequations}
Using operator matrices, this can also be written as 
\begin{align*}
	\begin{bmatrix}
		0 & 0\\
		\cD & \cC
	\end{bmatrix}\begin{bmatrix}
		\dot{u}\\\dot{p}
	\end{bmatrix} = \begin{bmatrix}
		-\cA & \phantom{-}\cD^{*}\\
		 0 & -\cB
	\end{bmatrix}\begin{bmatrix}
		u\\p
	\end{bmatrix} + \begin{bmatrix}
		f\\
		g
	\end{bmatrix},
\end{align*}
which directly reveals the PDAE structure. 
As a consequence, explicit methods for time integration cannot be used~\cite[Ch.~5.2]{KunM06}.

\subsection{BDF operators and weighted inner product spaces}
\label{sec:prelim:bdf}

Since we consider higher-order time discretizations of~\eqref{eq:ellpar} using \BDF,
we introduce the following notation: 
for a sequence $(y^n)_{n\in\Z}$ with $y^n\in\cX$ in a given Hilbert space $\cX$ and
a step size $\tau$, we define the \BDF-$k$ operator
\begin{equation}
	\label{eqn:BDFoperator}
	\bdf{k}{y^n} \vcentcolon=
	\frac{1}{\tau}\, \bigg(\sum_{\ell=0}^k \xi_\ell\, y^{n-\ell} \bigg),
\end{equation}
where the coefficients $\xi_\ell$ are determined from the relation 
\begin{equation}
	\label{eqn:BDFpolynomial}
	\xi(s) 
	\vcentcolon= \sum_{\ell=0}^k \xi_\ell s^\ell 
	= \sum_{\ell=1}^k \tfrac{1}{\ell}\, (1-s)^\ell;
\end{equation}
see for instance~\cite[Ch.~III.3]{HaiNW09} and \Cref{tab:coeffBDF} for some example
coefficients.
By abuse of notation, we ignore the dependency on $k$ and simply write
$\xi_\ell = \xi_{k,\ell}$. \BDF methods up to order $6$ are known to be
stable~\cite[Ch.~III.3]{HaiNW09}, but for the analysis we only consider
\BDF methods up to order $5$, for which simple multipliers can be constructed to
make it $A$--stable; see~\cite{NevO81} and the following proposition. 
\begin{table}
	\renewcommand{\arraystretch}{1.25}
	\caption{\BDF-$k$ coefficients~$\xi_\ell$, $\ell=0,\dots,k$, for different values of $k$.}
	\label{tab:coeffBDF}
	\centering
	\begin{tabular}{@{\quad}c@{\quad}|@{\quad}rrrrrr}
		$k$ & \quad$\xi_0$ & \quad$\xi_1$ & \quad$\xi_2$ & \quad$\xi_3$ & \quad$\xi_4$ & \quad$\xi_5$\\	\midrule
		$1$ & $1$ & $-1$& & & &\\
		$2$ & $\tfrac{3}{2}$ & $-2$ & $\tfrac{1}{2}$& & &\\
		$3$ & $\tfrac{11}{6}$ & $-3$ & $\tfrac{3}{2}$& $-\tfrac{1}{3}$& &\\
		$4$ & $\tfrac{25}{12}$ & $-4$ & $3$& $-\tfrac{4}{3}$&$\tfrac{1}{4}$ &\\
		$5$ & $\tfrac{137}{60}$ & $-5$ & $5$& $-\tfrac{10}{3}$ & $\tfrac{5}{4}$ & -$\tfrac{1}{5}$\\
	\end{tabular}
\end{table}

To simplify the forthcoming analysis of the iterative schemes involving
discretization via the \BDF operator~\eqref{eqn:BDFoperator}, we will use special
inner products by means of the following preliminary result, taken as a special case
from~\cite{NevO81,Dah78}, see also~\cite[Ch.~V.~6, Ch.~V.~8]{HaiW96}.
For this, we introduce for a matrix $G=[g_{ij}]\in\R^{k\times k}$, $k\geq 1$, an
operator $\calM\colon\cX\to\cX$, and a sequence $(y^n)_{n\in\Z}$ with $y^n\in\cX$
for $n\in\Z$ the Kronecker notation 
\begin{displaymath}
	(G\otimes \calM) \begin{bmatrix}
		y^{n}\\
		\vdots\\
		y^{n+1-k}
	\end{bmatrix} \vcentcolon= \begin{bmatrix}
		\sum_{j=1}^k g_{1j} \calM y^{n+1-j}\\
		\vdots\\
		\sum_{j=1}^k g_{kj} \calM y^{n+1-j}
	\end{bmatrix}
	\in \cX^{k}.
\end{displaymath}
With this, we obtain the following auxiliary result. 
\begin{proposition}
\label{prop:summation}
Let $\cX$ be a Hilbert space and $k\in\{1,\ldots,5\}$.
Then there exist a parameter~$\eta\in[0,1)$, a matrix $G = [g_{ij}]\in\spd{k}$, and
a vector $\gamma\in\R^{k+1}$ such that for any self-adjoint and monotone operator
$\calM\colon\cX\to\cX$, any sequence $(y^n)_{n\in\Z}$ with $y^n\in\cX$ for $n\in\Z$,
and constant $\tau>0$, we have
\begin{equation}
	\label{eqn:BDFtested}
	\begin{aligned}
		&\Big\langle \tau\calM\, \bdf{k}{y^n}, y^n - \eta y^{n-1}\Big\rangle_{\cX} 
		- \Big\|\calM^{1/2}\sum_{\ell=0}^k \gamma_{\ell+1}y^{n-\ell}\Big\|_{\cX}^2\\
		&\hspace{3em}= \left\langle (G\otimes \calM) \begin{bmatrix}
		y^{n}\\
		\vdots\\
		y^{n+1-k}
	\end{bmatrix},\begin{bmatrix}
		y^{n}\\
		\vdots\\
		y^{n+1-k}
	\end{bmatrix}\right\rangle_{\!\!\!\cX^{k}} -\ \left\langle (G\otimes \calM)
	\begin{bmatrix}
		y^{n-1}\\
		\vdots\\
		y^{n-k}
	\end{bmatrix},\begin{bmatrix}
		y^{n-1}\\
		\vdots\\
		y^{n-k}
	\end{bmatrix}\right\rangle_{\!\!\!\cX^{k}}.
	\end{aligned}
\end{equation}
\end{proposition}
\begin{remark}
The first term on the left hand side of~\eqref{eqn:BDFtested} is independent of
$\tau$, since $\bdf{k}{\cdot}$ scales with~$\frac{1}{\tau}$.
Hence, the above statement holds true independently of $\tau$.
\end{remark}
\begin{proof}[Proof of \Cref{prop:summation}]
	For a self-adjoint and monotone operator, the square root $\calM^{1/2}$ is well-defined and is itself a self-adjoint operator \cite[Ch.~5, \S 3, Th.~3.35]{Kat95}.
	With this, we consider the sequence $(z^n)_{n\in\Z}$, given by $z^n \vcentcolon= \calM^{1/2} y^n\in\cX$ for $n\in\Z$. Then \eqref{eqn:BDFtested} is equivalent to
	\begin{equation*}
		\begin{aligned}
			&\big\langle \tau\, \bdf{k}{z^n}, z^n - \eta z^{n-1}\big\rangle_{\cX} 
			- \Big\|\sum_{\ell=0}^k \gamma_{\ell+1}z^{n-\ell}\Big\|_{\cX}^2\\
			&\hspace{3em}= \left\langle G \begin{bmatrix}
			z^{n}\\
			\vdots\\
			z^{n+1-k}
		\end{bmatrix},\begin{bmatrix}
			z^{n}\\
			\vdots\\
			z^{n+1-k}
		\end{bmatrix}\right\rangle_{\!\!\!\cX^{k}} -\ \left\langle G \begin{bmatrix}
			z^{n-1}\\
			\vdots\\
			z^{n-k}
		\end{bmatrix},\begin{bmatrix}
			z^{n-1}\\
			\vdots\\
			z^{n-k}
		\end{bmatrix}\right\rangle_{\!\!\!\cX^{k}}.
		\end{aligned}
	\end{equation*}
	Note that this is an algebraic criterion, which appears in proving the equivalence
	between $G$- and $A$-stability of the multistep methods.
	We refer to~\cite[Ch.~V.6, Def.~6.3]{HaiW96} for the definition of $G$-stability
	and to~\cite[Ch.~V.1, Def.~1.1 and Th.~1.3]{HaiW96} for the
	definition of $A$-stability and the corresponding criterion.
	Together with the assumption that the first inner
	product on the left-hand side is at most zero, the above criterion leads to $G$-stability for multistep methods.
	The $A$-stability criterion for the \BDF-$k$ methods reads
	\begin{displaymath}
		\real{{\frac{\xi(\zeta)}{1 - \eta \zeta}}} \ge 0 
		\quad\text{for}\quad 
		|\zeta| \ge 1
	\end{displaymath}
	for appropriate multipliers $\eta$ and the polynomial $\xi$
	from~\eqref{eqn:BDFpolynomial}.
	The existence of $\eta \in [0,1)$ for $k\le 5$ is shown in~\cite{NevO81}.
	The remaining part of the proof is the same as that
	in~\cite[Ch.~V.6, Ch.~V.8]{HaiW96} albeit with an arbitrary inner product space
	instead of the finite-dimensional Euclidean space. 
\end{proof}
\begin{example}
	For \BDF-$1$ ($\xi_0 = 1, \xi_1 = -1$), which equals the implicit Euler scheme,
	we immediately see that equality~\eqref{eqn:BDFtested} is satisfied for the
	choice $\eta=0$, $G = \tfrac{1}{2}$,
	and $\gamma = \big[\pm\tfrac{1}{\sqrt{2}},\mp\tfrac{1}{\sqrt{2}}\big]$.
	For \BDF-$2$ with coefficients
	$\xi_0 = \tfrac{3}{2}$, $\xi_1 = -2$, $\xi_2 = \tfrac{1}{2}$, the choice
	\begin{equation*}
		\eta=0, \qquad 
		G = \begin{bmatrix}\tfrac{5}{4} & -\tfrac{1}{2}
			\\ -\tfrac{1}{2} & \tfrac{1}{4}
		\end{bmatrix},\qquad 
		\gamma = \pm\, \Big[\tfrac{1}{2}\ \ 1\ \ \tfrac{1}{2}\Big]
	\end{equation*}
	leads to equality~\eqref{eqn:BDFtested}, cf.~\cite[Ch.\,V.\,6, Ex.\,6.5]{HaiW96}.
\end{example}
Under the assumptions of \Cref{prop:summation}, the inner product used in~\eqref{eqn:BDFtested} defines a norm, which we denote as
\begin{align*}
	\Vert E \Vert_{G, \calM}^2 
	\vcentcolon= \big\langle (G\otimes\calM) E,E \big\rangle_{\cX^{k}}
\end{align*}
for all $E \in \cX^{k}$.
\subsection{Discrete Gr\"{o}nwall-type inequality}
\label{sec:prelim:gronwall}

	We use the following perturbed Gr\"{o}nwall-type inequality in
	the proof of \Cref{th:convergence} below.
	\begin{lemma}
		\label{lem:gronwalls}
		Consider sequences of positive numbers $\{a^{n}\}$, $\{b^{n}\}$,
		and $\{c^{n}\}$ satisfying
		\[
		a^{n} - (1 + \tau C)\, a^{n-1} + b^{n} - b^{n-1}
		\le c^{n}
		\]
		for some positive constants $\tau, C >0$. Then it holds that 
		\begin{displaymath}
			{a}^{n} + {b}^{n} \le e^{C n \tau} \Big({a}^{0} + {b}^{0}
			+ \frac{1}{C\tau}\max_{1\le\ell\le n}{c^{\ell}}\Big).
		\end{displaymath}
	\end{lemma}
	\begin{proof}
		The idea of the proof is to transform the perturbed	sum into a
		telescoping sum; see also~\cite{Emm99}.
		For this, we set ${\tilde a}^{n} \vcentcolon= (1 + \tau C)^{-n}a^{n}$
		and ${\tilde b}^{n} \vcentcolon= (1 + \tau C)^{-n}b^{n}$.
		Then the assumed inequality yields the estimate 
		\begin{align*}
			{\tilde a}^{n} - {\tilde a}^{n-1} 
			&= (1 + \tau C)^{-n} \big(a^{n} -(1 + \tau C)\, a^{n-1}\big)
			\nonumber\\
			&\le (1 + \tau C)^{-n}c^{n} - (1 + \tau C)^{-n}b^n
			+ (1 + \tau C)^{-n}b^{n-1} \nonumber\\
			&\le (1 + \tau C)^{-n}c^{n} - (1 + \tau C)^{-n}b^n
			+ (1 + \tau C)^{-(n-1)}b^{n-1} \nonumber\\
			&\le (1 + \tau C)^{-n}c^{n} - {\tilde b}^{n}
			+ {\tilde b}^{n-1},
		\end{align*}
		where we have used $(1 + \tau C)^{-n} \le (1 + \tau C)^{-{n-1}}$.
		Summing up from $1$ to $n$, we obtain
		\begin{align*}
			{\tilde a}^{n} + {\tilde b}^{n} \le {\tilde a}^{0} + {\tilde b}^{0}
			+ \sum_{\ell=1}^{n}(1 + \tau C)^{-\ell}c^{\ell}
			\le {\tilde a}^{0} + {\tilde b}^{0}
			+ \sum_{\ell=1}^{n}(1 + \tau C)^{-\ell}\max_{1\le\ell\le n}{c^{\ell}},
		\end{align*}
		which simplifies to
		\begin{align*}
			{a}^{n} + {b}^{n} \le (1 + \tau C)^{n} \Big({a}^{0} + {b}^{0}
			+ (1 + \tau C)^{-n}
				\sum_{\ell=0}^{n-1}(1 + \tau C)^{\ell}
				\max_{1\le\ell\le n}{c^{\ell}}\Big).
		\end{align*}
		With the observations 
		\begin{displaymath}
			(1 + \tau C)^{n} \le e^{C n \tau} \quad \text{and} \quad
			(1 + \tau C)^{-n}\sum_{\ell=0}^{n-1}(1 + \tau C)^{\ell} \le
			(1 + \tau C)^{-n}\frac{(1 + \tau C)^{n} - 1}{C\tau}
			\le \frac{1}{C\tau}
		\end{displaymath}
		the assertion follows.
	\end{proof}
%
\section{Iterative Decoupling and BDF Time Integration}
\label{sec:iterativeBDF}

To iteratively decouple \eqref{eq:ellpar:opt}, we consider sequences $(u^i)_{i\in\N}$, $(p^i)_{i\in\N}$ together with their derivatives and enforce 
\begin{align}
	\cD{\dot u}^{i} - \cD{\dot u}^{i-1} 
	&= L \cI ({\dot p}^{i} - {\dot p}^{i-1}) \qquad \,\text{in } \cQ^{*}, \label{eq:ep:fsdecoup}
\end{align}
where $\cI\colon\cQ^{*}\to\cQ^{*}$ is the identity operator, $L > 0$ a stabilization parameter, and the superscript denotes the (inner) iteration number. Including this equation in~\eqref{eq:ellpar:opt:b}, we obtain the decoupled system 
\begin{subequations}
	\label{eq:ep:fs}
		\begin{align}
			\cA u^{i} - \cD^{*} p^{i} 
			&= f \qquad \text{in } \cV^{*}, \label{eq:ep:fs:a}\\
			\cD{\dot u}^{i-1} + \cC{\dot p}^{i} + \cB p^{i} + L \cI ({\dot p}^{i} - {\dot p}^{i-1}) 
			&= g \qquad \,\text{in } \cQ^{*}. \label{eq:ep:fs:b}
		\end{align}
\end{subequations}
This system is indeed decoupled, since~\eqref{eq:ep:fs:b} can be used to first solve for~$p^{i}$. Afterwards, $u^{i}$ is solved from \eqref{eq:ep:fs:a}. System \eqref{eq:ep:fs} is well-posed and the iterates converge globally to the true solution, cf.~\cite{Muj22,MikW13}.  
We now consider a discretization in time and apply the~\BDF-$k$ scheme to~\eqref{eq:ep:fs}, leading to
\begin{subequations}
	\label{eq:ep:fs:bdf}
	\begin{align}
		\cA u^{n,i} - \cD^{*} p^{n,i} &= f^{n}  \qquad \text{in } \cV^{*}, \label{eq:ep:fs:bdf:a}\\
		\cD\bdf{k}{u^{n,i-1}} + \cC\bdf{k}{p^{n,i}} + \cB p^{n,i} + L \cI (\bdf{k}{p^{n,i}} - \bdf{k}{p^{n,i-1}})&= g^{n}  \qquad \text{in } \cQ^{*}. \label{eq:ep:fs:bdf:b}
	\end{align}
\end{subequations}
Here, $y^{n,i}$ denotes the approximation of $y(t^{n})$ after $i$ inner iterations (for $y\in\{u,p\}$), and the~\BDF-$k$ operator $\bdf{k}{y^{n,i}}$ is defined as
\begin{displaymath}
	\bdf{k}{y^{n,i}} 
	\vcentcolon= \frac{1}{\tau}\, \bigg(\xi_{0} y^{n,i} + \sum_{\ell=1}^k \xi_{\ell}\, y^{n-\ell,J_{n-\ell}} \bigg).
\end{displaymath}
Therein, $\xi_0,\ldots,\xi_k$ are the \BDF coefficients (see \Cref{sec:prelim:bdf}) and $J_n$ denotes the (minimal) number of iterations such that the error between two iterates for the solution at time point $t^{n}$ goes below a prescribed tolerance $\tol$. More precisely, we set~$J_n = i$, when
\begin{align}
	\label{eq:iterNumDef}
	\tfrac{c_a}{2}\, \Vert u^{n,i} - u^{n,i-1}\Vert_{\cV}^{2}
	+ \big(c_c + \tfrac{L}{2}\big)\, \Vert p^{n,i} - p^{n,i-1} \Vert^{2}_{\cHQ}
	+ \tfrac{\tau}{\xi_0}c_b\, \Vert p^{n,i} - p^{n,i-1} \Vert^{2}_{\cQ} 
	\le \tol^2
\end{align}
and stop the inner iterations.
The initial guess for the iterates can be chosen to be the solution at the previous
time step.
The choice of the norm on the left-hand side of \eqref{eq:iterNumDef} is motivated
from the following lemma. 
\begin{lemma}
\label{lem:contraction}
Let the stabilization parameter $L$ satisfy $L \ge \frac{C_a^2 C_d^2}{2c_a^3}$.
Then the discrete iterative problem \eqref{eq:ep:fs:bdf} is contractive and, hence, well-defined.
\end{lemma}
\begin{proof}
	Similar to the first-order in time case~\cite{StoBKNR19}, it can be easily shown that~\eqref{eq:ep:fs:bdf} is contractive for all $n \ge k$.
	To see this, define
	\begin{displaymath}
		\Delta_u^{n,i} \vcentcolon= u^{n,i} - u^{n,i-1} \quad\text{and}\quad \Delta_p^{n,i} \vcentcolon= p^{n,i} - p^{n,i-1}. 
	\end{displaymath}
	Considering the difference between two iterates of~\eqref{eq:ep:fs:bdf} and multiplying the last of the resulting second subsystem by $\tfrac{\tau}{\xi_0}$, we get 
	\begin{subequations}
		\label{eq:contrac:0}
		\begin{align}
			\cA \Delta_u^{n,i} - \cD^{*} \Delta_p^{n,i} &= 0  \qquad \text{in } \cV^{*}, 	\label{eq:contrac:0:a}\\
			\cD\Delta_u^{n,i-1} + \cC \Delta_p^{n,i} + \tfrac{\tau}{\xi_0}\cB \Delta_p^{n,i} + L \cI (\Delta_p^{n,i} - \Delta_p^{n,i-1})&= 0 \qquad \text{in } \cQ^{*}. \label{eq:contrac:0:b}
		\end{align}
	\end{subequations}
	With the choice of test functions $\Delta_u^{n,i}$ in~\eqref{eq:contrac:0:a} and $\Delta_p^{n,i}$ in~\eqref{eq:contrac:0:b} and adding the resulting two equations, we obtain 
	\begin{align*}
		\langle \cA \Delta_u^{n,i}, \Delta_u^{n,i}\rangle + \langle \cC \Delta_p^{n,i}, \Delta_p^{n,i}\rangle + \tfrac{\tau}{\xi_0}\langle \cB& \Delta_p^{n,i}, \Delta_p^{n,i}\rangle + L \langle \Delta_p^{n,i} - \Delta_p^{n,i-1}, \Delta_p^{n,i}\rangle \\
		&\quad= \langle \cD^{*} \Delta_p^{n,i},\Delta_u^{n,i}-\Delta_u^{n,i-1}\rangle \\
		&\quad= \langle \cA \Delta_u^{n,i},\Delta_u^{n,i}-\Delta_u^{n,i-1}\rangle.
	\end{align*}
	With the assumed properties of the bilinear forms, it follows together with the weighted Young's inequality with an arbitrary positive constant $\delta > 0$ that 
	\begin{equation*}
		\begin{aligned}
			c_a\, \Vert \Delta_u^{n,i}\Vert_{\cV}^{2} + c_c\,\Vert &\Delta_p^{n,i} \Vert^{2}_{\cHQ} + \tfrac{\tau}{\xi_0}c_b\, \Vert \Delta_p^{n,i} \Vert^{2}_{\cQ} \nonumber\\
			&+ \tfrac{L}{2}\, \Big(\Vert\Delta_p^{n,i}\Vert_{\cHQ}^{2} - \Vert\Delta_p^{n,i-1}\Vert_{\cHQ}^{2} + \Vert\Delta_p^{n,i}-\Delta_p^{n,i-1}\Vert_{\cHQ}^{2} \Big) \\
			&\qquad\qquad\le C_a \Vert\Delta_u^{n,i}\Vert_{\cV}\Vert\Delta_u^{n,i}-\Delta_u^{n,i-1}\Vert_{\cV}\nonumber\\
			&\qquad\qquad\le \frac{\delta}{2}\, \Vert\Delta_u^{n,i}\Vert_{\cV}^{2} + C_a^2\frac{1}{2\delta}\, \Vert\Delta_u^{n,i}-\Delta_u^{n,i-1}\Vert_{\cV}^{2}.
		\end{aligned}
	\end{equation*}
	From \eqref{eq:contrac:0:a}, we get 
	\begin{align*}
		\Vert \Delta_u^{n,i} \Vert^{2}_{\cV} 
		\le \tfrac{C_d^{2}}{c_a^{2}}\, \Vert \Delta_p^{n,i} \Vert^{2}_{\cHQ}
		\qquad\text{and}\qquad
		\Vert \Delta_u^{n,i} - \Delta_u^{n,i-1} \Vert^{2}_{\cV} 
		\le \tfrac{C_d^{2}}{c_a^{2}}\, \Vert \Delta_p^{n,i} - \Delta_p^{n,i-1}\Vert^{2}_{\cHQ}.
	\end{align*}
	Combining the above estimates yields
	\begin{align}
		\label{contrac:arbit}
		\big(c_a -\tfrac{\delta}{2}\big)\, \Vert \Delta_u^{n,i}\Vert_{\cV}^{2}
		+ \big(&c_c + \tfrac{L}{2}\big)\, \Vert \Delta_p^{n,i} \Vert^{2}_{\cHQ}
		+ \tfrac{\tau}{\xi_0}c_b\, \Vert \Delta_p^{n,i} \Vert^{2}_{\cQ} \nonumber\\
		&+ \big(\tfrac{L}{2} - \tfrac{1}{2\delta}\tfrac{C_a^2 C_d^2}{c_a^2}\big)\,
		\Vert\Delta_p^{n,i}-\Delta_p^{n,i-1}\Vert_{\cHQ}^{2} 
		\le \tfrac{L}{2} \Vert \Delta_p^{n,i-1} \Vert^{2}_{\cHQ}.
	\end{align}
		We need $\delta < 2 c_a$ and $L > \frac{C_a^2 C_d^2}{2c_a^3}$ to ensure
		positivity of the terms on the left-hand side.
		Choosing $\delta = c_a$ and adding positive terms on the right-hand side
		gives
		\begin{align}
			\label{contrac:arbit:1}
			&\Big(
			\tfrac{c_a}{2}\,\Vert \Delta_u^{n,i}\Vert_{\cV}^{2}
			+ \big(c_c + \tfrac{L}{2}\big)\, \Vert \Delta_p^{n,i} \Vert^{2}_{\cHQ}
			+ \tfrac{\tau}{\xi_0}c_b\, \Vert \Delta_p^{n,i} \Vert^{2}_{\cQ}
			\Big)
			+ \tfrac{1}{2} \big(L - \tfrac{C_a^2 C_d^2}{c_a^3}\big)\,
			\Vert\Delta_p^{n,i}-\Delta_p^{n,i-1}\Vert_{\cHQ}^{2}\nonumber\\
			&\quad\qquad\le \frac{\tfrac{L}{2}}{c_c + \tfrac{L}{2}}\ \Big(
			\tfrac{c_a}{2}\,
			\Vert \Delta_u^{n,i-1}\Vert_{\cV}^{2}
			+ \big(c_c + \tfrac{L}{2}\big) \Vert \Delta_p^{n,i-1} \Vert^{2}_{\cHQ}
			+ \tfrac{\tau}{\xi_0}c_b\, \Vert \Delta_p^{n,i-1} \Vert^{2}_{\cQ}	
			\Big).
		\end{align}	
		Choosing the stabilization parameter $L = \frac{C_a^2 C_d^2}{c_a^3}$,
		we obtain 
		\begin{align*}
			\mathfrak{d}\big((u^{n,i},p^{n,i}),(u^{n,i-1},p^{n,i-1})\big) 
			\le\gamma\, \mathfrak{d}\big((u^{n,i-1},p^{n,i-1}),(u^{n,i-2},p^{n,i-2})\big)
		\end{align*}
		where
		\begin{displaymath}
			\gamma^2 = \frac{\tfrac{L}{2}}{c_c + \tfrac{L}{2}} = \frac{C_a^2 C_d^2}{C_a^2 C_d^2 + 2 c_a^3 c_c} < 1
		\end{displaymath}
		and
		\begin{displaymath}
			\mathfrak{d}\big((u^{n,i},p^{n,i}),(u^{n,i-1},p^{n,i-1})\big)
			\vcentcolon= \Big(\tfrac{c_a}{2}\Vert \Delta_u^{n,i}\Vert_{\cV}^{2}
			+ \big(c_c + \tfrac{L}{2}\big) \Vert \Delta_p^{n,i} \Vert^{2}_{\cHQ}
			+ \tfrac{\tau}{\xi_0}c_b\, \Vert \Delta_p^{n,i} \Vert^{2}_{\cQ}\Big)^{1/2}.
			\qedhere
		\end{displaymath}
\end{proof}
Let $\epsilon^{n,1} = \mathfrak{d}((u^{n,1},p^{n,1}),(u^{n,0},p^{n,0}))$ denote the norm of the error at time point $t^{n}$ due to the initial guess for solving \eqref{eq:ep:fs:bdf}. From the termination criterion \eqref{eq:iterNumDef} and the contraction, we have 
\begin{align*}
	\gamma^{J_n -1} \epsilon^{n,1} \le \tol
\end{align*}
such that the number of needed inner iterations is given by
\begin{equation}
	\label{eq:numIterations}
	J_n = \left\lceil \frac{\ln{\tol} - \ln{\epsilon^{n,1}}}{\ln{\gamma}}\right\rceil + 1.
\end{equation}
Minimizing the computational work goes along with minimizing $J_n$.
An optimal choice balances the errors corresponding to the operator splitting and the time discretization. To realize this, we need to know the relevant rates. This is subject of the following theorem, which is the main result of this article. Therein, we show that the iterates of~\eqref{eq:ep:fs:bdf} converge to the fixed-point, which is the exact solution of~\eqref{eq:ellpar:opt}.
\begin{theorem}[convergence result]
	\label{th:convergence}
	Let $u \in \cC^{k}([0,T],\cV)$, $p \in \cC^{k}([0,T],\cHQ)$ be the solution of
	\eqref{eq:ellpar} for sufficiently smooth right-hand sides~$f$ and~$g$.
	Consider $k\in\{1,\ldots,5\}$ and given initial data
	$(u^\ell,p^\ell) \in\cV\times\cHQ$ for $\ell=0,\ldots,k-1$.
	Let $(u^{n,J_{n}}, p^{n,J_{n}})$ for $n\ge k$ be the solution of
	\eqref{eq:ep:fs:bdf} where $J_n$ denotes the number of inner iterations at the
	time point $t^{n}$ as defined in \eqref{eq:iterNumDef} for a
	prescribed tolerance~$\tol > 0$. Then, it holds that
	\begin{align}
		\label{eqn:iterativeBDFconv}
		\big\Vert u^{n,J_{n}} - u(t^{n}) \big\Vert_{\cV}^{2} + \big\Vert p^{n,J_{n}} - p(t^{n}) \big\Vert_{\cHQ}^{2} 
		\lesssim \frac{\tol^{2}}{\tau^{3}} + \tau^{2k} + \sum_{\ell=0}^{k-1}\Vert p^{\ell} - p(t^{\ell})\Vert_{\cHQ}^{2}
	\end{align}
	with a hidden constant including a factor of the form $e^{C t_n}$ with some constant $C>0$. 
\end{theorem}
\begin{proof}
We define the errors of the iterates with respect to the true solution and the defects on approximating the true derivative with \BDF-$k$, namely
\begin{gather*}
	e_{u}^{n,i}\vcentcolon= u^{n,i} - u(t^{n}), \qquad \tau d_{u}^{n}\vcentcolon= \xi_{0} u(t^{n}) + \cdots + \xi_{k} u(t^{n-k}) - \tau {\dot u}(t^{n}),\\
	e_{p}^{n,i}\vcentcolon= p^{n,i} - p(t^{n}), \qquad \tau d_{p}^{n}\vcentcolon= \xi_{0} p(t^{n}) + \cdots + \xi_{k} p(t^{n-k}) - \tau {\dot p}(t^{n}).
\end{gather*}
An application of the Taylor series shows that the defects satisfy the estimates
\begin{displaymath}
	\Vert d_u^{n} \Vert_{\cV} \lesssim \tau^{k}, \qquad \Vert d_p^{n} \Vert_{\cHQ} \lesssim \tau^{k}.
\end{displaymath}
Now consider the difference between~\eqref{eq:ep:fs:bdf} and~\eqref{eq:ellpar:opt}. Making use of
\begin{align*}
	\bdf{k}{y^{n,i}} - \bdf{k}{y^{n,i-1}} 
	&= \tfrac{\xi_0}{\tau}\, \big( y^{n,i} - y^{n,i-1} \big) 
\end{align*}
for $y\in\{u,p\}$, one can show that the errors and the defects satisfy the system 
\begin{subequations}
	\label{eq:ep:fs:err}
	\begin{align}
		\cA e_{u}^{n,i} - \cD^{*} e_{p}^{n,i} 
		&= 0 \qquad \text{in } \cV^{*}, \label{eq:ep:fs:err:a}\\
		\cD\big(\tau\bdf{k}{e_{u}^{n,i}}\big) 
		+ \xi_{0}\cD \big( e_{u}^{n,i-1} - e_{u}^{n,i}\big) 
		+ \cC\big(\tau\bdf{k}{e_{p}^{n,i}}\big) \hspace*{2em}&\nonumber\\
		+ \tau \cB e_{p}^{n,i} + \xi_0 L \cI \big(e_{p}^{n,i} - e_{p}^{n,i-1} \big) + \tau \cD d_u^{n} + \tau \cC d_p^{n}
		&= 0 \qquad \text{in } \cQ^{*}. \label{eq:ep:fs:err:b}
	\end{align}
\end{subequations}
Testing \eqref{eq:ep:fs:err:a} with $e_{u}^{n,i} \in \cV$ and using the properties of the operators, it is easy to see that
\begin{align}
	\label{eq:fs:err:est:0}
	\Vert e_{u}^{n,i} \Vert_{\cV} \le \tfrac{C_d}{c_a} \Vert e_{p}^{n,i} \Vert_{\cHQ}.
\end{align}
Note that \eqref{eq:ep:fs:err:a} holds for all $i>0$, hence solving for~$e_{u}^{n,i}\in\cV$ and $e_{u}^{n,i-1}\in\cV$ and inserting this into~\eqref{eq:ep:fs:err:b} gives
\begin{equation*}
	\begin{aligned}
		(\cC + \cD \cA^{-1}\cD^*)\tau\bdf{k}{e_{p}^{n,i}}
		+ (\xi_0 L \cI - \xi_{0}\cD \cA^{-1}\cD^*) \big( e_{p}^{n,i} - e_{p}^{n,i-1}\big) \hspace*{4em}&\nonumber\\+ \tau \cB e_{p}^{n,i} + \tau \cD d_u^{n} + \tau \cC d_p^{n}
		= 0 \qquad \text{in } \cQ^{*}.
	\end{aligned}
\end{equation*}
Let $\cM\vcentcolon= \cC + \cD \cA^{-1}\cD^{*}$ and note that $\cD \cA^{-1}\cD^{*}$ is self-adjoint and monotone.
Therefore, together with the Riesz isomorphism in~$\cHQ$, the operator~$\cM\colon\cHQ\to\cHQ$ is self-adjoint and elliptic with ellipticity constant $c_c$.
Testing with $e_{p}^{n,i} - \eta e_{p}^{n-1,J_{n-1}} \in \cQ$, where $\eta$ is the appropriate multiplier from \Cref{prop:summation} such that the algebraic criterion \eqref{eqn:BDFtested} is satisfied for some~$G \in\spd{k}$, we get
\begin{align}
	\label{eq:fs:err:0}
	T_1 + T_2 = T_3 + T_4,
\end{align}
where
\begin{align*}
	T_1 &\vcentcolon= \tau\, \big\langle \cM\, \bdf{k}{e_{p}^{n,i}}, e_{p}^{n,i} - \eta e_{p}^{n-1,J_{n-1}} \big\rangle,\\ 
	T_2 &\vcentcolon= \tau\, \big\langle \cB e_{p}^{n,i}, e_{p}^{n,i} - \eta e_{p}^{n-1,J_{n-1}} \big\rangle,\\
	T_3 &\vcentcolon= \big\langle \xi_{0}(\cD \cA^{-1}\cD^* -L \cI )\big(e_{p}^{n,i} - e_{p}^{n,i-1}\big), e_{p}^{n,i} - \eta e_{p}^{n-1,J_{n-1}} \big\rangle,\\
	T_4 &\vcentcolon= -\tau\, \big\langle \cD d_{u}^{n} + \cC d_{p}^{n}, e_{p}^{n,i} - \eta e_{p}^{n-1,J_{n-1}} \big\rangle.
\end{align*}
Define the vector of error terms
\begin{displaymath}
	\renewcommand{\arraystretch}{1.5}
	E_{p}^{n,i} \vcentcolon= \begin{bmatrix}
		e_{p}^{n,i}\\
		e_{p}^{n-1,J_{n-1}}\\
		\vdots\\
		e_{p}^{n+1-k,J_{n+1-k}}
	\end{bmatrix} \in \cHQ^{k},
\end{displaymath}
where the vector takes the same superscript as that of the first error.
Using \Cref{prop:summation}, we have
\begin{subequations}
	\label{eq:fs:err:est}
	\begin{align}
		T_1 &\ge \Vert E_{p}^{n,i} \Vert_{G,\cM}^{2} - \Vert E_{p}^{n-1,J_{n-1}} \Vert_{G,\cM}^{2}.
	\end{align}
	By the ellipticity of~$\calB$, we have
	\begin{align}
		2\,T_2 
		&= 2\tau\, \big\langle \cB e_{p}^{n,i}, e_{p}^{n,i} - \eta\, e_{p}^{n-1,J_{n-1}} \big\rangle \nonumber\\
		&= \tau\, \|e_{p}^{n,i}\|^2_b - \tau\, \| \eta\, e_{p}^{n-1,J_{n-1}}\|^2_b + \tau\, \| e_{p}^{n,i} - \eta\, e_{p}^{n-1,J_{n-1}} \|^2_b \nonumber\\
		&\ge \tau\, \|e_{p}^{n,i}\|^2_b - \tau\eta^2\, \| e_{p}^{n-1,J_{n-1}}\|^2_b.
	\end{align}
	For the third term, we apply the continuity of the appearing operator and the weighted Young inequality with a positive constant $\delta > 0$, leading to 
	\begin{align}	
		T_3 &\le \xi_{0}\big(\tfrac{C_d^2}{c_a} + L\big)\Vert e_{p}^{n,i} - e_{p}^{n,i-1}\Vert_{\cHQ} \Vert e_{p}^{n,i}\Vert_{\cHQ} \nonumber\\
		&\qquad+ \eta\,\xi_{0}\big(\tfrac{C_d^2}{c_a} + L\big)\Vert e_{p}^{n,i} - e_{p}^{n,i-1}\Vert_{\cHQ} \Vert e_{p}^{n-1,J_{n-1}}\Vert_{\cHQ} \nonumber\\
		&\le \tau\frac{\delta}{2}\Vert e_{p}^{n,i}\Vert^{2}_{\cHQ} + \frac{1}{2\tau\delta}\xi_{0}^{2}\big(\tfrac{C_d^2}{c_a} + L\big)^{2}\Vert e_{p}^{n,i} - e_{p}^{n,i-1}\Vert^{2}_{\cHQ} \nonumber\\
		&\qquad + \tau\frac{1}{2}\Vert e_{p}^{n-1,J_{n-1}}\Vert^{2}_{\cHQ} + \frac{1}{2\tau}\eta^2\xi_{0}^2\big(\tfrac{C_d^2}{c_a} + L\big)^{2}\Vert e_{p}^{n,i} - e_{p}^{n,i-1}\Vert^{2}_{\cHQ}.
	\end{align}
	Finally, a similar estimate for the fourth term yields	
	\begin{align}
		T_4 &\le \tau\, C_d\Vert d_{u}^{n}\Vert_{\cV} \Vert e_{p}^{n,i}\Vert_{\cHQ} + \tau\, C_d\,\eta\, \Vert d_{u}^{n}\Vert_{\cV} \Vert e_{p}^{n-1,J_{n-1}}\Vert_{\cHQ} \nonumber\\
		& \qquad+ \tau\, C_c\Vert d_{p}^{n}\Vert_{\cHQ} \Vert e_{p}^{n,i}\Vert_{\cHQ} + \tau\, C_c\,\eta\, \Vert d_{p}^{n}\Vert_{\cHQ} \Vert e_{p}^{n-1,J_{n-1}}\Vert_{\cHQ} \nonumber\\
		&\le \tau\frac{\delta}{2}\Vert e_{p}^{n,i}\Vert^{2}_{\cHQ} + \tau\frac{1}{2}\Vert e_{p}^{n-1,J_{n-1}}\Vert^{2}_{\cHQ}\nonumber\\
		&\qquad  + \tau\,\Big(\frac{1}{\delta} + \eta^2\Big) \big(C_d^2\, \Vert d_{u}^{n}\Vert^{2}_{\cV} + C_c^2\,\Vert d_{p}^{n}\Vert^{2}_{\cHQ}\big). 
	\end{align}
	The combination of~\eqref{eq:fs:err:est} and~\eqref{eq:fs:err:0} gives 
	\begin{align*}
		\Vert E_{p}^{n,i} &\Vert_{G,\cM}^{2} - \Vert E_{p}^{n-1,J_{n-1}} \Vert_{G,\cM}^{2} + \tau\tfrac{1}{2}\, \|e_{p}^{n,i}\|^2_b - \tau\tfrac{1}{2}\eta^2\, \| e_{p}^{n-1,J_{n-1}}\|^2_b\nonumber\\
		&\le \tau \delta \Vert e_{p}^{n,i} \Vert_{\cHQ}^{2} + \tau \Vert e_{p}^{n-1,J_{n-1}} \Vert_{\cHQ}^{2} 
		+ \tau\,\Big(\frac{1}{\delta} + \eta^2\Big) \big(C_d^2\, \Vert d_{u}^{n}\Vert^{2}_{\cV} + C_c^2\, \Vert d_{p}^{n}\Vert^{2}_{\cHQ}\big)\nonumber\\
		&\qquad +\frac{1}{2\tau}\Big(\frac{1}{\delta} + \eta^2\Big)\xi_{0}^{2}\big(\tfrac{C_d^2}{c_a} + L\big)^{2}\Vert e_{p}^{n,i} - e_{p}^{n,i-1}\Vert^{2}_{\cHQ}.
	\end{align*}
\end{subequations}
To absorb the first two terms on the right-hand side, observe that 
\begin{align*}
	\Vert e_{p}^{n,i} \Vert_{\cHQ}^{2} \le
	\frac{\CQtoHsquare}{c_b} \Vert e_{p}^{n,i} \Vert_{b}^{2}
	\quad\text{and}\quad
	c_G \Vert e_{p}^{n-1,J_{n-1}} \Vert_{\cHQ}
	\le \Vert E_{p}^{n-1,J_{n-1}} \Vert_{G,\cM}
\end{align*}
where $c_G$ is the smallest eigenvalue of $G$.
For the last term, we use
\begin{displaymath}
	\Vert e_{p}^{n,i} - e_{p}^{n,i-1}\Vert^{2}_{\cHQ}
	\le \CQtoH \Vert e_{p}^{n,i} - e_{p}^{n,i-1}\Vert^{2}_{\cQ}
\end{displaymath}
and \eqref{eq:iterNumDef}, i.e, 
with the consideration of $i=J_{n}$ inner iteration steps, we obtain
\begin{align*}
	\Vert E_{p}^{n,J_{n}} \Vert_{G,\cM}^{2}
	- &\big(1+\tfrac{1}{c_G}\tau\big)\Vert E_{p}^{n-1,J_{n-1}} \Vert_{G,\cM}^{2}
	+ \tau\,\big(\tfrac{1}{2} - \delta \tfrac{\CQtoHsquare}{c_b}\big)
	\, \|e_{p}^{n,J_{n}}\|^2_b
	- \tfrac{1}{2}\tau\eta^2\, \| e_{p}^{n-1,J_{n-1}}\|^2_b\nonumber\\
	&\le \frac{1}{2}\Big(\frac{1}{\delta} + \eta^2\Big)\xi_{0}^{2}
	\big(\tfrac{C_d^2}{c_a} + L\big)^{2}\frac{\tol^2}{\tau^{2}} 
	+ \,\Big(\frac{1}{\delta} + \eta^2\Big) \big(C_d^2\,
	+ C_c^2\, \big)\tau^{2k + 1}.
\end{align*}
Now, we set 
\begin{displaymath}
	\delta = \frac{1}{2}\frac{c_b}{\CQtoHsquare} \big(1 - \eta^2\big)
\end{displaymath}
in order to make the last two terms on the left hand side amenable for telescoping
sum. 
Using the discrete Gr\"{o}nwall inequality from \Cref{lem:gronwalls} yields
the estimate 
\begin{align*}
	\Vert E_{p}^{n,J_{n}} \Vert_{G,\cM}^{2} + \tau\, \Vert e_{p}^{n,J_{n}}\Vert^2_b \lesssim e^{C t_n} \Big(\frac{\tol^2}{\tau^{3}} + \tau^{2k} + \Vert p^{k-1} - p(t^{k-1})\Vert^2_b + \sum_{\ell=0}^{k-1}\Vert p^{\ell} - p(t^{\ell})\Vert_{\cHQ}^{2}\Big).
\end{align*}
With this and~\eqref{eq:fs:err:est:0}, assertion~\eqref{eqn:iterativeBDFconv}
follows.
\end{proof}
The theorem states that the overall error is bounded by the initial error, a
term of order~$\tau^k$, and a term of order $\tol/\tau^{3/2}$.
Assuming sufficiently good initial data, the optimal choice for the tolerance is
hence~$\tol = C\,\tau^{k+3/2}$ for some positive constant~$C$.
We will investigate this result in more detail numerically in \Cref{sec:numerics:balancing}.
\begin{remark}
	\label{rem:implicit}
	In the derivation of the above error estimate, if we drop the index for the iteration and hence not use the contraction condition, we recover the derivation of error estimates for the fully implicit time discretization.
	In more detail, let $u^{n}, p^{n}$ denote the solution of fully implicit time discretization. Then
	\begin{align*}
		\big\Vert u^{n} - u(t^{n}) \big\Vert_{\cV}^{2} + \big\Vert p^{n} - p(t^{n}) \big\Vert_{\cHQ}^{2} 
		\lesssim \tau^{2k} + \sum_{\ell=0}^{k-1}\Vert p^{\ell} - p(t^{\ell})\Vert_{\cHQ}^{2}
	\end{align*}
	with same hidden constants as in \Cref{th:convergence} only depending on the \BDF method and the operators of the original problem~\eqref{eq:ellpar:opt}.
\end{remark}
\begin{remark}
	\label{rem:modifyThm}
	The above proof does not require the operator $\cB$ to be self-adjoint.
	In particular, the above result holds true also for the multiple network case
	considered in \Cref{exp:network} where the operator $\cB$ is modified to include
	the network exchange terms with additional assumptions to make it
	elliptic, cf. \cite[Lem.~2.3]{AltMU21c}.
\end{remark}
%
%
\section{Numerical Experiments}
\label{sec:numerics}
In this section, we gather a number of numerical experiments illustrating
the convergence results of the paper derived in
\Cref{sec:numerics:orders,sec:numerics:balancing}.
Moreover, we do a parametric study of the number of inner iteration steps in
\Cref{sec:numerics:iterations} and apply the proposed decoupling strategy to a
three-dimensional brain model in \Cref{sec:numerics:brain}.
We use the FEniCSx~\cite{BarDDHHRRSSW23} finite element library for the experiments
and use direct solvers (standard LU decomposition from the PETSc library) for all the experiments.
%
\subsection{Convergence studies}
\label{sec:numerics:orders}

In this first experiment, we consider a manufactured two-dimensional problem similar
to the one in \cite{ErnM09} but with homogeneous Dirichlet boundary conditions.
On the unit square~$\Omega=(0,1)^2$ we set 
\begin{align*}
	u(t,x,y) = -10 \mathrm{e}^{-\tfrac{1}{5} t}\sin(\pi x) \sin(\pi y)
	\begin{bmatrix}1\\1\end{bmatrix}, 
	\quad p(t,x,y) = 10 \mathrm{e}^{-\tfrac{1}{5} t}\sin(\pi x) \sin(\pi y).
\end{align*}
The poroelasticity parameters are chosen as in \Cref{tab:convergenceParameters}.
\begin{table}
    \centering
    \caption{Physical parameters for the two-dimensional poroelasticity convergence study.}
    \label{tab:convergenceParameters}
	\begin{tabular}{@{\quad}c@{\qquad\qquad}c@{\qquad\qquad}c@{\qquad\qquad}c@{\qquad\qquad}c@{\quad}}
		\toprule
		$\lambda$ & $\mu$         & $\tfrac{\kappa}{\nu}$ & $\tfrac{1}{M}$ & $\alpha$   \\ \midrule
		$0.5$ & $0.125$ & $0.05$          & $4$      & $0.75$ \\
		\bottomrule
	\end{tabular}
\end{table}
To demonstrate the convergence of the (decoupled) \BDF-$k$ methods for
$1 \le k \le 5$, we need to ensure that the errors due to the spatial 
discretization are very small.
Therefore, we choose continuous $(P_6, P_5)$ finite elements for $(u,p)$ with mesh
size $h=2^{-6}$.
The manufactured solution is also used to set the needed initial data for all the
multistep methods.

The tolerance~$\tol$ is chosen sufficiently small, i.e., $10^{-3-3k}$ for each $1 \le k \le 5$, respectively.
The maximum of the errors over a finite number of time points in the interval
$[0,10]$ are plotted in \Cref{fig:errorComparToy} and clearly indicate the expected rates. 
As a benchmark, the errors are compared with respect to fully coupled methods, i.e.,
implicit \BDF-$k$ schemes applied to the fully coupled problem (see \Cref{rem:implicit}). 
\begin{figure}[ht]
	\centering
	\begin{subfigure}[t]{.45\linewidth}
		\input{conv_u_toy_H1.tex}
	\end{subfigure}
	\hspace{-0.95cm}
	\begin{subfigure}[t]{.45\linewidth}
		\input{conv_p_toy_L2.tex}
	\end{subfigure}\\
	\ref*{legToy} 
	\caption{Errors for \BDF-$k$ for $1 \le k \le 5$ for a
	fixed spatial mesh with mesh size~$h=2^{-6}$.
	The solid blue curves correspond to the newly introduced decoupled
	methods whereas the dashed red curves correspond to the fully coupled methods
	(original \BDF schemes).
	Left: $\cV$-error in the displacement~$u$.
	Right: $\cHQ$-error in the pressure~$p$.}
	\label{fig:errorComparToy}
\end{figure}
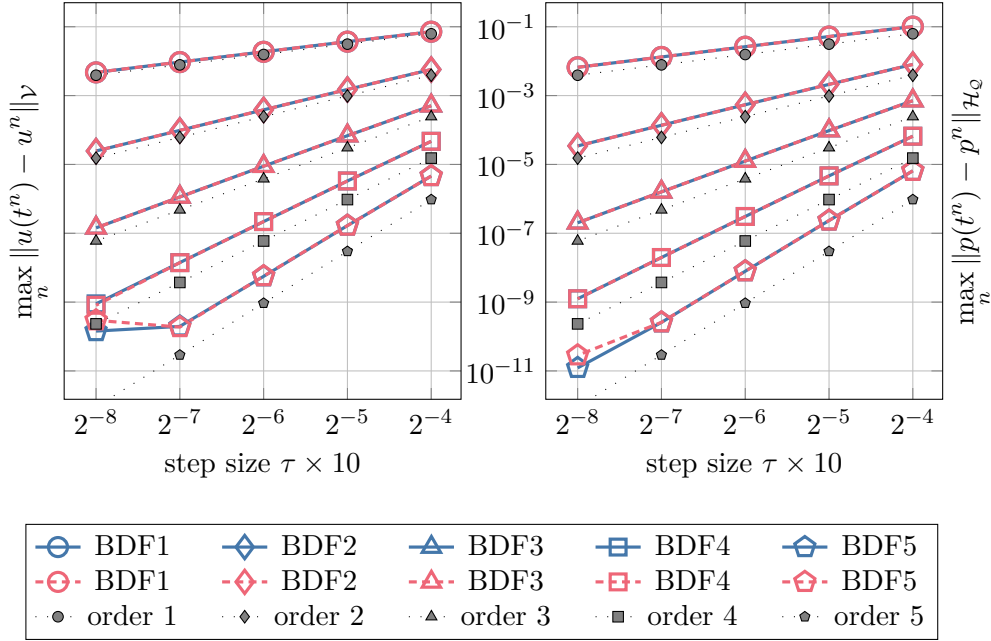
%
\subsection{Error balancing}
\label{sec:numerics:balancing}
As already discussed in the end of \Cref{sec:iterativeBDF}, the optimal choice for
$\tol$ in order to obtain a balanced overall error is~$\tol = C\,\tau^{k+3/2}$.
To verify this choice with $C=1$, we consider the same manufactured problem as in
\Cref{sec:numerics:orders} but with $T=1$.
Here, however, for the spatial discretization, we pick $(P_{k+2}, P_{k+1})$
finite elements for the pair $(u,p)$ with $h=2^{-7}$ which is sufficient
to keep the spatial error components small.
Let
\begin{displaymath}
	\text{error} 
	= \max_{n} \Big(\Vert u(t^{n}) - u^{n}\Vert_{\cV} + \Vert p(t^{n}) - p^{n}\Vert_{\cHQ}\Big)
\end{displaymath}
denote the maximum of the errors between a discrete solution: either fully implicit or iterative
(with the index of iteration dropped) and the true solution computed on a finite number of time points.
In \Cref{fig:expSearchC}, the value of this $error$ is plotted for the choices
$\tol= \tau^{s}$, where $s\in \{k, k +\tfrac{1}{2}, k+1, k +\tfrac{3}{2}, k+2\}$.
Exemplary for \BDF-$1$, \BDF-$2$, we see that this prediction indeed
holds true.
When the tolerances are bigger there is no convergence.
Whenever the $\tol$ is at most $\tau^{k+3/2}$, convergence is observed and the errors are close to the errors which arise for a fully implicit method.
For \BDF-$3$, we do not see this sharpness. 

In \Cref{tab:averageIters}, the average number of inner iterations $J_n$
per time step is reported.
When either the time step size decreases or the order of the method increases, then
the error that can be reached becomes smaller.
This is reflected in a growing number of inner iterations required to reach these
potentially smaller errors in the time discretization.
Furthermore, picking a too small tolerance results in unnecessary iterations.
One more parameter that helps in controlling the number of iterations is the
stabilization parameter $L$, whose optimality is discussed in~\cite{StoBKNR19}.
The true optimal choice, however, depends on the elasticity operator, the fluid
operators, and the time step size.
We investigate this in the next subsection with a toy problem.
\begin{figure}
	\centering
	\hspace{-2em}
	\begin{subfigure}[t]{.3\linewidth}
		\input{itertol_error_BDF1.tex}
	\end{subfigure}
	\begin{subfigure}[t]{.3\linewidth}
		\input{itertol_error_BDF2.tex}
	\end{subfigure}
	\begin{subfigure}[t]{.3\linewidth}
		\input{itertol_error_BDF3.tex}
	\end{subfigure}
	\ref*{legITERTOLC} 
	\caption{Errors for \BDF-$1$ (left), \BDF-$2$ (middle), and \BDF-$3$
	(right) for a fixed spatial mesh size~$h=2^{-7}$.
	The gray line indicate orders $1$ (left), $2$ (middle), and $3$ (right). }
	\label{fig:expSearchC}
\end{figure}
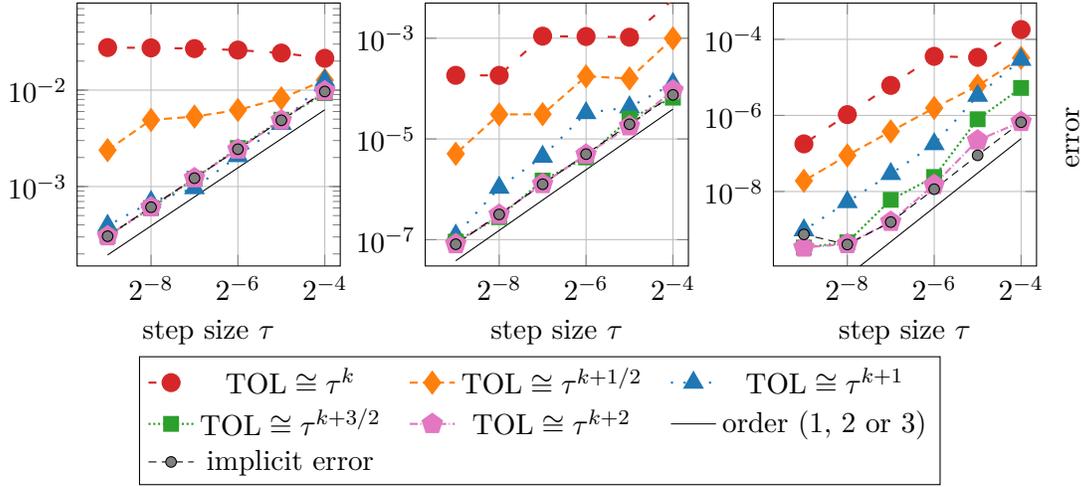
\begin{table}
	\caption{Average number of inner iteration steps $J_n$ per time step.}
	\label{tab:averageIters}
	\sisetup{minimum-decimal-digits=1,exponent-mode = fixed, fixed-exponent = 0,round-mode=places,round-precision=1}
	\begin{tabular}{l@{\quad}l|@{\qquad}l@{\qquad}l@{\qquad}l@{\qquad}l@{\qquad}l@{\qquad}l}
		\toprule
	\multicolumn{1}{l}{}  & \hspace{1.3cm}$\tau$         & $2^{-4}$ & $2^{-5}$ & $2^{-6}$ & $2^{-7}$ & $2^{-8}$ & $2^{-9}$ \\\midrule
	\multirow{3}{*}{\BDF-$1$} & $\tol = \tau^{k+1}$ 	& \num{5}	& \num{5}	& \num{5}	& \num{6}		& \num{6}& \num{7}\\
						   & $\tol = \tau^{k+3/2}$ 		& \num{5}	& \num{6}	& \num{7}	& \num{7}		& \num{8}& \num{8}\\
						   & $\tol = \tau^{k+2}$ 		& \num{6}	& \num{7}	& \num{8}	& \num{8.7}		& \num{9}& \num{10}\\\midrule
	\multirow{3}{*}{\BDF-$2$} & $\tol = \tau^{k+1}$		& \num{6}	& \num{7}	& \num{8}	& \num{8.83}	& \num{9} & \num{10}\\
						   & $\tol = \tau^{k+3/2}$ 		& \num{7}	& \num{8}	& \num{9}	& \num{10}		& \num{11}& \num{12}\\
						   & $\tol = \tau^{k+2}$ 		& \num{8}	& \num{9}	& \num{10}	& \num{11}		& \num{12.1}& \num{14}\\\bottomrule
	\end{tabular}
\end{table}
\subsection{Needed iteration steps to reach \texorpdfstring{$\tol$}{TOL}}
\label{sec:numerics:iterations}
To study the sharpness of the estimates for the number of 
iterations~\eqref{eq:numIterations} as a function of the stabilization parameter,
we consider the scalar toy problem  
\begin{displaymath}
	a(u,v) = v^{\top}\! A u, \qquad 
	d(v,p) = \sqrt{\omega}\, p D v, \qquad 
	c(p,q) = q C p, \qquad 
	b(p,q) = q B p,
\end{displaymath}
where
\begin{align*}
	A = \frac{1}{2-\sqrt{2}} \begin{bmatrix}
		2 & -1 & 0 \\
		-1 & 2 & -1\\
		0 & -1 & 2
	\end{bmatrix}, \qquad
	B = 1, \qquad
	C = 1, \qquad
	D = \begin{bmatrix}
		\frac{2}{3} & \frac{1}{3} & \frac{2}{3}
	\end{bmatrix}.
\end{align*}
As right-hand sides, we set	
\begin{align*}
	f(t) \equiv \begin{bmatrix}
		1 \\ 1 \\ 1
	\end{bmatrix}, \qquad
	g(t) = 100 \sin(t).
\end{align*}
In this case, system~\eqref{eq:contrac:0} simplifies to
\begin{align}
	\label{eq:toycontraction}
	\Delta_p^{n,i} 
	= \gamma \Delta_p^{n,i-1}
	\qquad\text{with}\qquad
	\gamma 
	= \frac{L - \omega D\,A^{-1}D^{\top}}{L + C + \tfrac{\tau}{\xi_0} B}.
\end{align}
For a stronger contraction, one can prescribe a smaller contraction 
constant $\gamma$, but then the stabilization parameter has to be chosen accordingly.
Rewriting~\eqref{eq:toycontraction}, we see that the stabilization parameter should be set as  
\begin{equation}
	\label{eq:stabParam}
	L = \frac{\omega}{1-\gamma}\, D\,A^{-1}D^{\top}
		+ \frac{\gamma}{1 - \gamma}\, \big(C + \tfrac{\tau}{\xi_0} B\big).
\end{equation}
For \BDF-$1$ and \BDF-$2$ and a given coupling strength $\omega$,
we study the average number of inner iterations (per time step) needed as a function of the time step size $\tau$ and the contraction constant $\gamma$.
This is reported in \Cref{tab:iterCounts}. 
For this toy example, the absolute tolerance $\tol$ for the methods is set based on the average error achieved by a fully implicit \BDF method of the same order, which happens to be different for different $\omega$.
The initial data for \BDF-$2$ is computed from the implicit \BDF-$1$ method, which is locally accurate up to order two.

We make the following  observations.
First, we observe less inner iterations per time step if the contraction constant~$\gamma$ is decreased, which yields a corresponding stabilization parameter $L$ as in~\eqref{eq:stabParam}. Second, if we fix $\gamma$ but increase the order of the method, also the number of inner iteration increases. Third, if the time step size is decreased, then the number of inner iterations increases. This is due to the choice of~$\tol$. Finally, as the stabilization parameter depends on the coupling strength $\omega$, the number of inner iterations are almost identical for different $\omega$. All observations are as expected from the previous discussion.
\begin{table}
	\caption{Average number of inner iterations per time step needed as a function of the time step size $\tau$ for different methods, coupling strengths~$\omega$, and contraction constants~$\gamma$. The resulting stabilization parameter $L$ is given by~\eqref{eq:stabParam}.}
	\label{tab:iterCounts}
	\begin{tabular}{@{\quad}l@{\quad}c@{\quad}l@{\qquad}|@{\qquad}l@{\qquad}l@{\qquad}l@{\qquad}l@{\qquad}l@{\qquad}l@{\quad}}
		\toprule
						    \multicolumn{2}{l}{}        & \hspace{2em}$\tau$                  & $2^{-3}$ & $2^{-4}$ & $2^{-5}$ & $2^{-6}$ & $2^{-7}$ & $2^{-8}$ \\\midrule
	\multirow{4}{*}{\BDF-1} & \multirow{2}{*}{$\omega=2$} & $\gamma = 1/2$  & $9$      & $13$     & $13$                         & $16$     & $19$     & $22$                         \\
						   &                             & $\gamma = 1/10$ & $5$      & $6$      & $7$                          & $7$      & $8$      & $9$                          \\\cmidrule{2-9}
						   & \multirow{2}{*}{$\omega=4$} & $\gamma = 1/2$  & $9$      & $12$     & $15$                         & $17$     & $20$     & $23$                         \\
						   &                             & $\gamma = 1/10$ & $5$      & $6$      & $7$                          & $7$      & $8$      & $9$                          \\ \midrule
	\multirow{4}{*}{\BDF-2} & \multirow{2}{*}{$\omega=2$} & $\gamma = 1/2$  & $13$     & $15$     & $18$                         & $22$     & $26$     & $30$                         \\
						   &                             & $\gamma = 1/10$ & $6$      & $7$      & $8$                          & $9$      & $11$     & $12$                         \\\cmidrule{2-9}
						   & \multirow{2}{*}{$\omega=4$} & $\gamma = 1/2$  & $12$     & $16$     & $20$                         & $24$     & $28$     & $32$                         \\
						   &                             & $\gamma = 1/10$ & $6$      & $7$      & $8$                          & $9$      & $10$     & $12$ \\\bottomrule
	\end{tabular}
\end{table}

\subsection{Biomechanics multiple-network problem}
\label{sec:numerics:brain}

To illustrate the tuning of decoupling for practical applications, we pick a network
problem -- see \Cref{exp:network} -- with three fluid pressures from biomechanics
(similar to the one in \cite{EliRT23}) and apply a \BDF-$2$ method.
The three fluid networks are the so-called \emph{arteriole}, \emph{venous}, and
\emph{perivascular} networks and the fluid is allowed to flow from the arteriole to
the venous and perivascular networks.
In \Cref{lem:contraction}, we have $\gamma^2=\frac{L}{2c_c + L}$ for a single pressure network.
For the network case, we set $c_c$ to be the minimum over $c_c$ of each fluid network.
Since \BDF-$2$ is a multistep method, we pick \BDF-$1$ (which is second-order
accurate locally) to initialize the needed data.
For the computational domain, a processed coarse brain mesh is taken
from~\cite{LeePMR19,PieR18}, which has $99\,605$ cells and $29\,037$ vertices;
see \Cref{fig:brainMesh:mesh}.
The boundary is such that $\partial\Omega = \partial\Omega_s \cup \partial\Omega_v$,
where $\partial\Omega_s$ is the boundary near the skull and $\partial\Omega_v$ is
the inner boundary that encloses the ventricle.

For the source terms, we set $g_1(t,x) = \tfrac{1}{2}(1 - \cos(2\pi t))$,
$g_2 \equiv g_3 \equiv 0$, and $f \equiv 0$.
The initial data for all fields are assumed to be zero.
Finally, the boundary conditions read
\begin{align*}
	u &= 0 &&\text{on}~\partial\Omega_s, &  
	(\sigma(u) - {\textstyle\sum_{j=1}^{3}}\alpha_j p_j)\cdot \mathbf{n} 
	&= 10 \mathbf{n} &&\text{on}~\partial\Omega_v,\\
	\nabla p_1\cdot \mathbf{n} &= 0 &&\text{on}~\partial\Omega, &
	p_2 &= 0  &&\text{on}~\partial\Omega, \\
	p_3 &= 10  &&\text{on}~\partial\Omega.
\end{align*}
with $\mathbf{n}$ denoting the outer unit normal vector.
The material parameters are given in \Cref{tab:networkParameters} in SI units
(see~\cite{EliRT23} for units).
We use P$1$ finite elements for all the variables.
Solving using a direct solver over the interval $[0,1]$ with $64$ time steps,
the fields at the final time are plotted in \Cref{fig:brainFields}.
With a reference solution computed using the fully implicit \BDF-$2$ method with $\tau=2^{-9}$, the convergence history is shown in \Cref{fig:brainMesh:conv}.

\begin{figure}
	\centering
	\begin{subfigure}[t]{0.42\linewidth}
		\includegraphics[width=\linewidth]{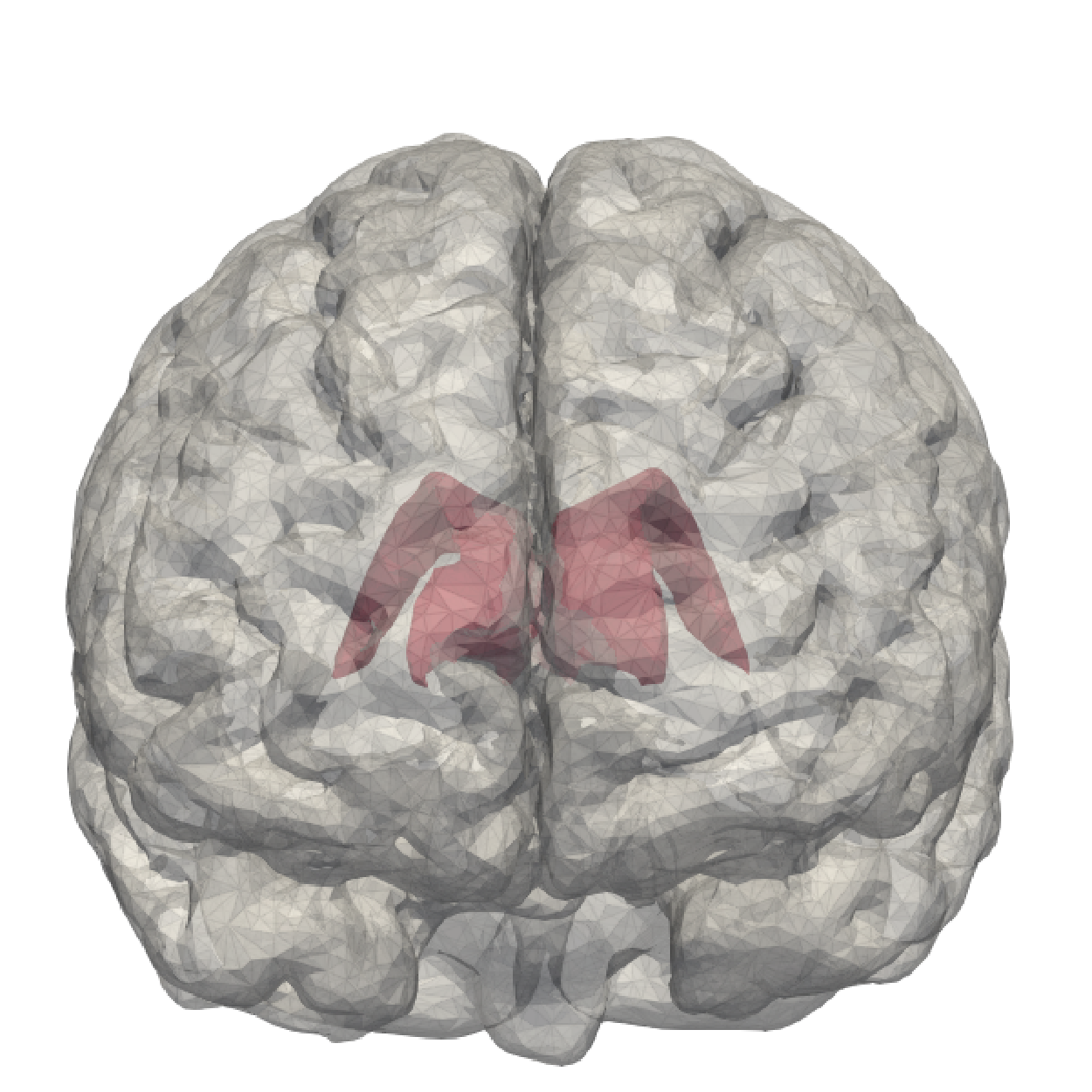}
		\caption{Computational mesh}
		\label{fig:brainMesh:mesh}
	\end{subfigure}
	\hfill
	\begin{subfigure}[t]{0.56\linewidth}
		\input{brain_conv.tex}
		\caption{Convergence history}
		\label{fig:brainMesh:conv}
	\end{subfigure}
	\caption{\textbf{(a)} The red part corresponds to ventricular boundary
	$\partial\Omega_v$, whereas the outer boundary is $\partial\Omega_s$.
	\textbf{(b)} Convergence history for all four fields.}
	\label{fig:brainMesh}
\end{figure}

\begin{table}
	\centering
	\caption{Parameters for the three-dimensional multiple-network problem.}
	\label{tab:networkParameters}
	\begin{tabular}{c|r@{\hspace{0.3em}}l@{\quad}r@{\hspace{0.3em}}l@{\quad}r@{\hspace{0.3em}}l}
		& \multicolumn{2}{l}{arteriole network} & \multicolumn{2}{l}{venous network} &  \multicolumn{2}{l}{perivascular network}\\\midrule
		$\lambda$\hspace{0.25cm}		= \num{9.08e+3}	& $\kappa_1/\nu_1=$	& \num{3.74e-8}	& $\kappa_2/\nu_2=$	& \num{3.74e-8} & $\kappa_3/\nu_3=$	& \num{1.43e-7}\\
		$\mu$\hspace{0.25cm}		= \num{5.48e+2}	& $1/M_1=$ 			& \num{2.90e-4}	& $1/M_2=$ 			& \num{1.50e-5} & $1/M_3=$ 			& \num{2.90e-4}\\
		$\beta_{12}$ 	= \num{1.0e-3}	& $\alpha_1=$ 		& \num{0.4} 	& $\alpha_2=$		& \num{0.2} 	& $\alpha_3=$		& \num{0.4}\\
		$\beta_{13}$ 	= \num{1.0e-4} 
	\end{tabular}
\end{table}

\begin{figure}[ht]
	\centering
	\begin{subfigure}[t]{.48\linewidth}
		\centering
		\includegraphics[width=\linewidth]{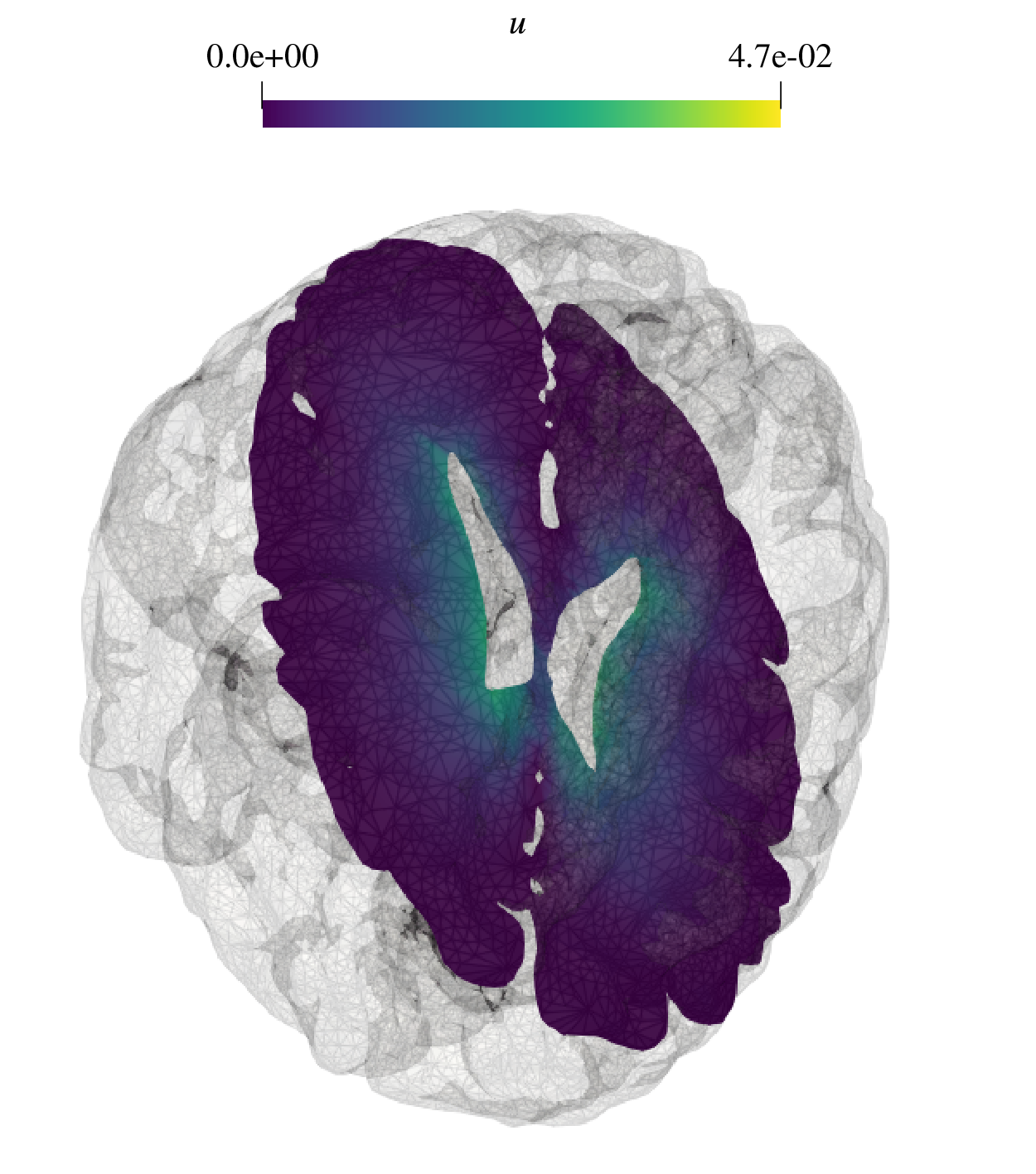}
		\caption{Displacement $u$}
	\end{subfigure}\hfill
	\begin{subfigure}[t]{.48\linewidth}
		\centering
		\includegraphics[width=\linewidth]{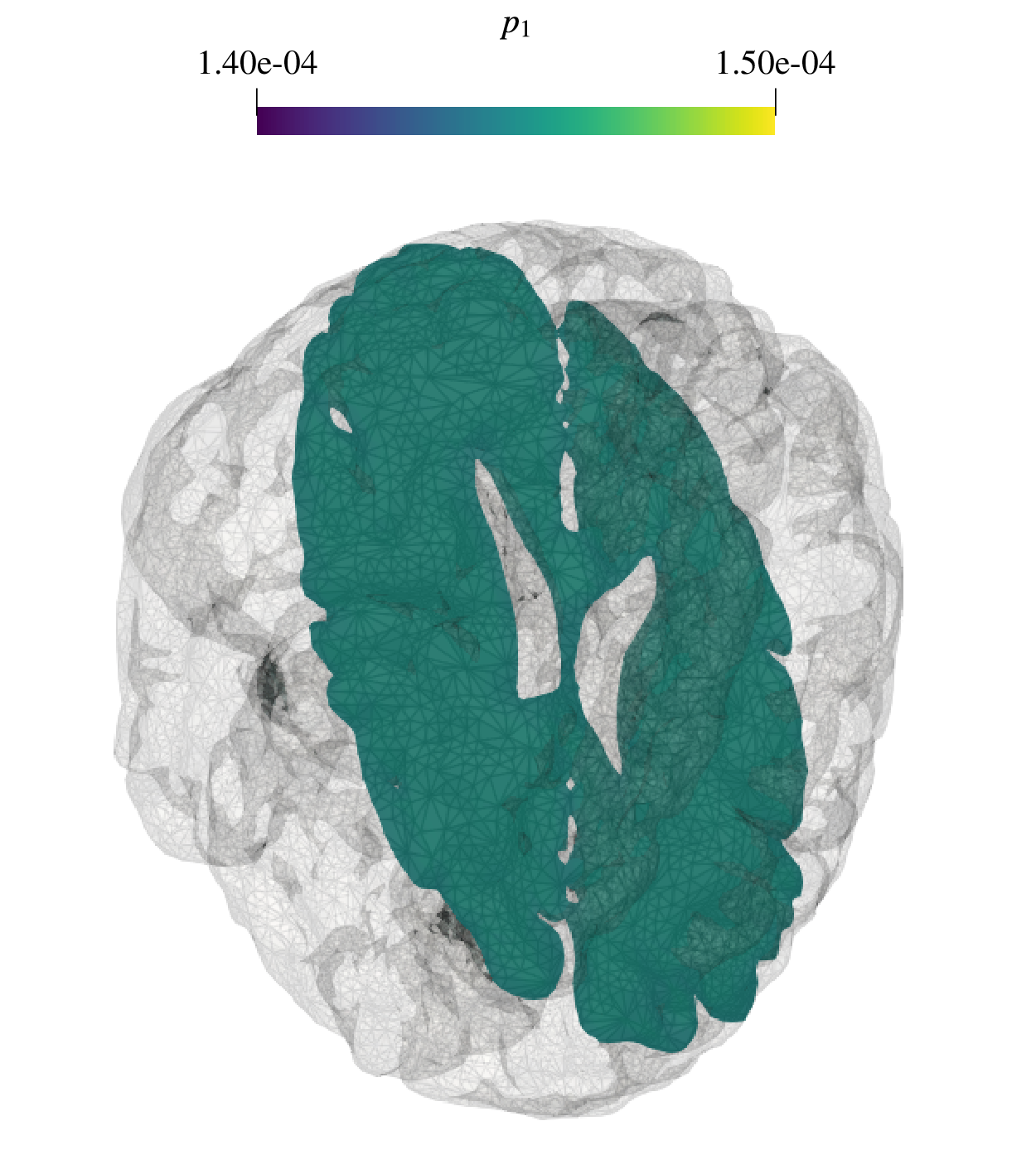}
		\caption{Arteriole pressure $p_1$}
	\end{subfigure}\\
	\begin{subfigure}[t]{.48\linewidth}
		\centering
		\includegraphics[width=\linewidth]{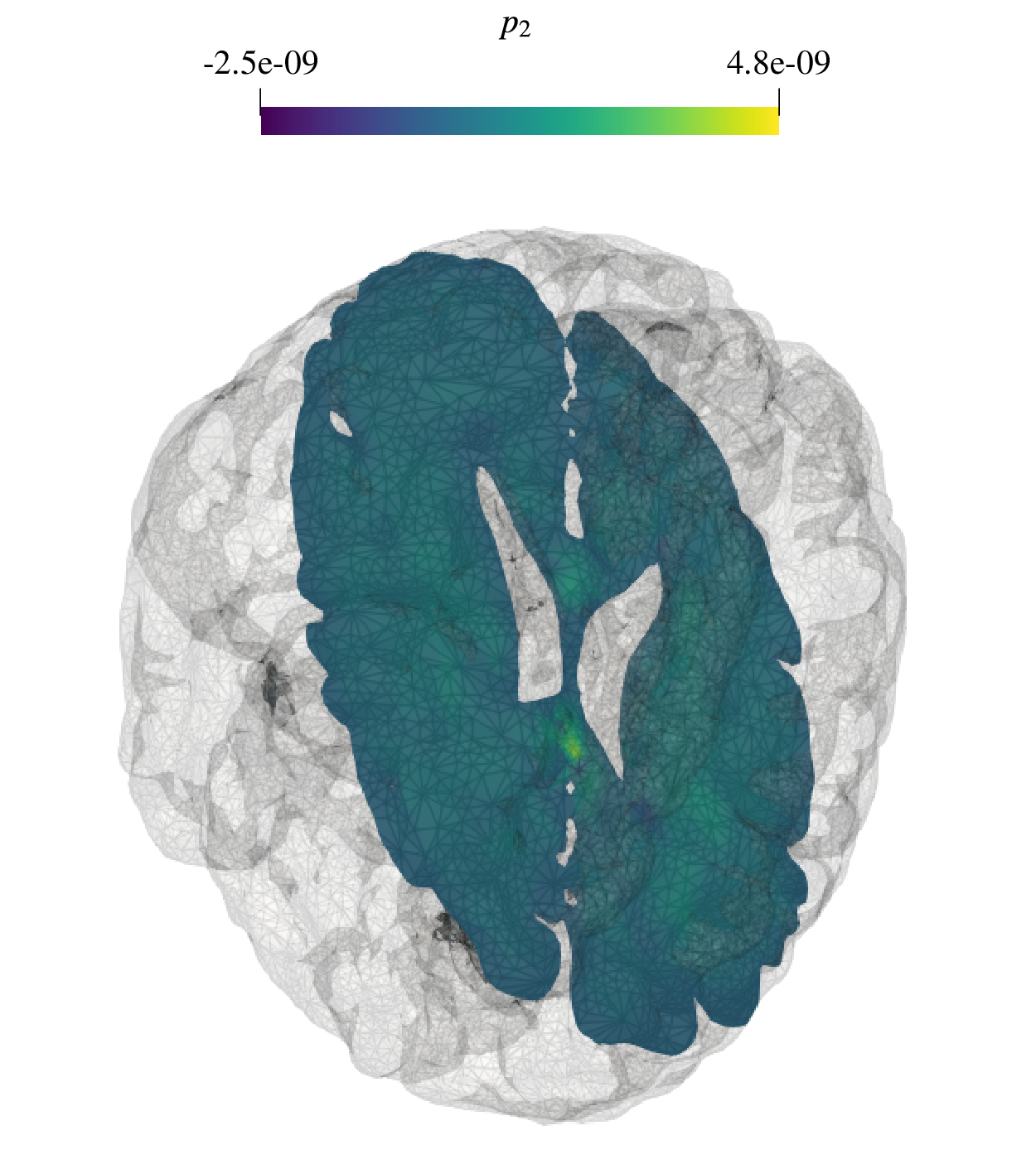}
		\caption{Venous pressure $p_2$}
	\end{subfigure}\hfill
	\begin{subfigure}[t]{.45\linewidth}
		\centering
		\includegraphics[width=\linewidth]{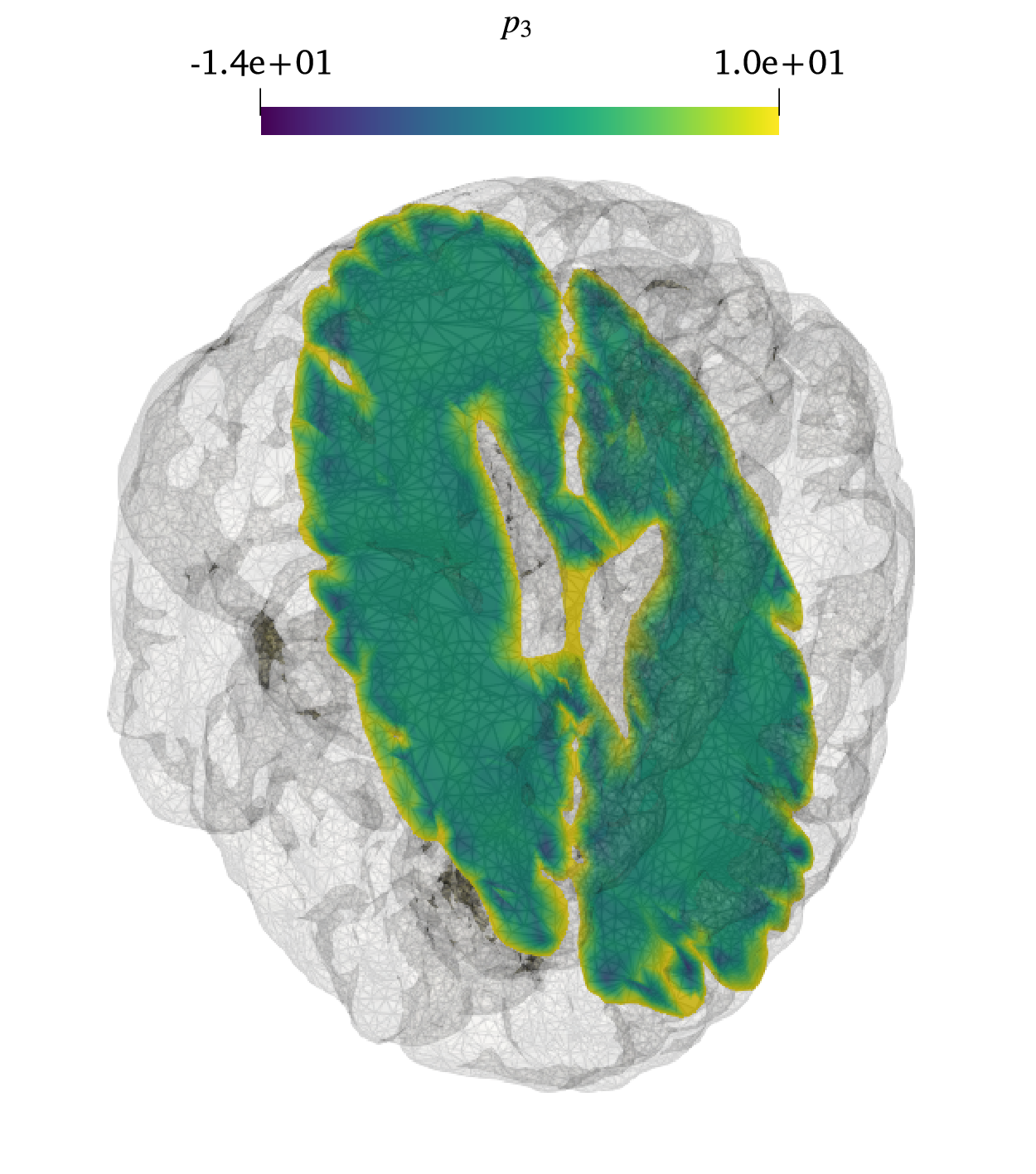}
		\caption{Perivascular pressure $p_3$}
	\end{subfigure}
	\caption{Illustration of the deformation and the three pressure variables at
	the final time $T=1$.}
	\label{fig:brainFields}
\end{figure}

\section{Conclusions}
\label{sec:conclusion}
In this article, a combined analysis of operator splitting together with higher-order \BDF time discretization is studied.
For an optimal balancing of the two error components, one needs to specify the tolerance for the termination criterion of the operator splitting in the specified form.
The results are thoroughly demonstrated in a number of numerical experiments and applied to a real-world example from biomechanics.	
A similar combined analysis of operator splitting together with Runge--Kutta methods is deferred for a future work.
%
%
\section*{Acknowledgements} 
This project is funded by the Deutsche Forschungsgemeinschaft (DFG, German Research Foundation) - 467107679. 
Parts of this work were carried out while RA and AM were affiliated with the Institute of Mathematics, University of Augsburg. RA was also affiliated with the Centre for Advanced Analytics and Predictive Sciences (CAAPS). Moreover, AM and BU acknowledge support by the Stuttgart Center for Simulation
Science (SimTech).
%
%
\bibliographystyle{plain-doi}
\bibliography{literature}

\end{document}

%% file: def.tex
\newcommand{\N}{\ensuremath\mathbb{N}}
\newcommand{\Z}{\ensuremath\mathbb{Z}}

\newcommand{\R}{\ensuremath\mathbb{R}}

\newcommand{\spd}[1]{\mathbb{S}_{\succ}^{#1}}

\newcommand{\T}{\ensuremath\mathsf{T}}

\newcommand{\dx}{\,\mathrm{d}x}

\newcommand{\hook}{\ensuremath{\hookrightarrow}}

\DeclareMathOperator{\id}{id}

\DeclareMathOperator{\real}{Re}
\DeclareMathOperator{\tol}{TOL}

\newcommand{\calB}{\mathcal{B}}

\newcommand{\calM}{\mathcal{M}}

\newcommand{\calX}{\mathcal{X}}

\newcommand{\cA}{\ensuremath{\mathcal{A}}}
\newcommand{\cB}{\ensuremath{\mathcal{B}}}
\newcommand{\cC}{\ensuremath{\mathcal{C}}}
\newcommand{\cD}{\ensuremath{\mathcal{D}}}

\newcommand{\cH}{\ensuremath{\mathcal{H}}}
\newcommand{\cI}{\ensuremath{\mathcal{I}}}

\newcommand{\cM}{\ensuremath{\mathcal{M}}}

\newcommand{\cQ}{\ensuremath{\mathcal{Q}}}

\newcommand{\cV}{\ensuremath{\mathcal{V}}}

\newcommand{\cX}{\ensuremath{\mathcal{X}}}

\newcommand{\cHV}{\ensuremath{\cH_{\scalebox{.5}{\cV}}}}
\newcommand{\cHQ}{\ensuremath{\cH_{\scalebox{.5}{\cQ}}}}

\newcommand{\cHQdual}{\ensuremath{\cH^*_{\scalebox{.5}{\cQ}}}}

\newcommand{\CQtoH}{\ensuremath{{C_{\scalebox{.5}{\cQ \hook \cHQ}}}} }
\newcommand{\CQtoHsquare}{\ensuremath{{C^2_{\scalebox{.5}{\cQ \hook \cHQ}}}} }

\definecolor{color0}{rgb}{1.0, 0.0, 0.0}
\definecolor{color1}{rgb}{0.0, 0.0, 1.0}
\definecolor{color2}{rgb}{0.1, 0.2, 0.9}
\definecolor{color3}{rgb}{0.9, 0.2, 0.1}
\definecolor{color4}{rgb}{0.8, 0.2, 0.8}
\definecolor{color5}{rgb}{0,0,0}

\definecolor{rabred}{rgb}{1.0, 0.0, 0.0}
\definecolor{rabblue}{rgb}{0.0, 0.0, 1.0}
\definecolor{rabcol}{rgb}{0.1, 0.2, 0.9}

\definecolor{crimson2143940}{RGB}{214,39,40}
\definecolor{darkgray176}{RGB}{176,176,176}
\definecolor{darkorange25512714}{RGB}{255,127,14}
\definecolor{forestgreen4416044}{RGB}{44,160,44}
\definecolor{gray127}{RGB}{127,127,127}
\definecolor{mediumpurple148103189}{RGB}{148,103,189}
\definecolor{orchid227119194}{RGB}{227,119,194}
\definecolor{sienna1408675}{RGB}{140,86,75}
\definecolor{steelblue31119180}{RGB}{31,119,180}

\definecolor{cbsblue}{RGB}{68,119,170}
\definecolor{cbscyan}{RGB}{102,204,238}
\definecolor{cbsgreen}{RGB}{34,136,51}
\definecolor{cbsyellow}{RGB}{204,187,68}
\definecolor{cbsred}{RGB}{238,102,119}
\definecolor{cbspurple}{RGB}{170,51,119}
\definecolor{cbsgrey}{RGB}{187,187,187}

\pgfplotscreateplotcyclelist{methodcompare}{%
color=cbsblue, every mark/.append style={scale=1.8, solid, fill=cbsblue}, very thick, mark=o\\
color=cbsblue, every mark/.append style={scale=2.2, solid, fill=cbsblue}, very thick, mark=diamond\\
color=cbsblue, every mark/.append style={scale=2, solid, fill=cbsblue}, very thick, mark=triangle\\
color=cbsblue, every mark/.append style={scale=1.5, solid, fill=cbsblue}, very thick, mark=square\\
color=cbsblue, every mark/.append style={scale=2, solid, fill=cbsblue}, very thick, mark=pentagon\\
densely dashed, color=cbsred, every mark/.append style={scale=1.8, solid, fill=cbsred}, very thick, mark=o\\
densely dashed, color=cbsred, every mark/.append style={scale=2.2, solid, fill=cbsred}, very thick, mark=diamond\\
densely dashed, color=cbsred, every mark/.append style={scale=2, solid, fill=cbsred}, very thick, mark=triangle\\
densely dashed, color=cbsred, every mark/.append style={scale=1.5, solid, fill=cbsred}, very thick, mark=square\\
densely dashed, color=cbsred, every mark/.append style={scale=2, solid, fill=cbsred}, very thick, mark=pentagon\\
loosely dotted, every mark/.append style={solid, fill=gray}, mark=*\\%
loosely dotted, every mark/.append style={scale=1.2, solid, fill=gray}, mark=diamond*\\%
loosely dotted, every mark/.append style={scale=1.2, solid, fill=gray}, mark=triangle*\\%
loosely dotted, every mark/.append style={solid, fill=gray}, mark=square*\\%
loosely dotted, every mark/.append style={solid, fill=gray}, mark=pentagon*\\%
}

\pgfplotscreateplotcyclelist{toomuchdata}{%
semithick, color=forestgreen4416044\\
semithick, color=darkorange25512714\\
semithick, color=gray127\\
semithick, color=mediumpurple148103189\\
semithick, color=orchid227119194\\
semithick, color=sienna1408675\\
semithick, color=steelblue31119180\\
semithick, color=crimson2143940\\
}

\pgfplotscreateplotcyclelist{errresnorm}{%
thick, color=forestgreen4416044\\
thick, color=darkorange25512714\\
thick, color=gray127\\
semithick, color=forestgreen4416044\\
semithick, color=darkorange25512714\\
semithick, color=gray127\\
}

\pgfplotscreateplotcyclelist{rab}{%
dotted, color=color0, every mark/.append style={scale=2, solid, fill=color0}, mark=o\\
dashdotted, color=color0, every mark/.append style={scale=2, solid, fill=color1}, mark=o\\
solid, color=color0, every mark/.append style={scale=2, solid, fill=gray},mark=o\\
dotted, color=color1, every mark/.append style={scale=2, solid, fill=color0},mark=triangle\\
dashdotted, color=color1, every mark/.append style={scale=2, solid, fill=color1},mark=triangle\\
solid, color=color1, every mark/.append style={scale=2, solid, fill=gray},mark=triangle\\
dotted, every mark/.append style={solid, fill=gray},mark=*\\
dashdotted, every mark/.append style={solid, fill=gray},mark=square*\\
solid, every mark/.append style={solid, fill=gray}, mark=triangle*\\
}

\pgfplotscreateplotcyclelist{iterativeconv}{%
color=forestgreen4416044, every mark/.append style={scale=1.8, solid, fill=forestgreen4416044}, thick, mark=o\\
color=forestgreen4416044, every mark/.append style={scale=2.2, solid, fill=forestgreen4416044}, thick, mark=diamond\\
color=forestgreen4416044, every mark/.append style={scale=2, solid, fill=forestgreen4416044}, thick, mark=triangle\\
color=forestgreen4416044, every mark/.append style={scale=1.5, solid, fill=forestgreen4416044}, thick, mark=square\\
color=forestgreen4416044, every mark/.append style={scale=2, solid, fill=forestgreen4416044}, thick, mark=pentagon\\
loosely dashed, every mark/.append style={solid, fill=gray}, mark=*\\%
densely dashed, every mark/.append style={scale=1.2, solid, fill=gray}, mark=diamond*\\%
loosely dotted, every mark/.append style={scale=1.2, solid, fill=gray}, mark=triangle*\\%
densely dotted, every mark/.append style={solid, fill=gray}, mark=square*\\%
densely dashdotted,every mark/.append style={solid, fill=gray}, mark=pentagon*\\%
loosely dashed, color=steelblue31119180, every mark/.append style={scale=1.8, solid, fill=steelblue31119180}, thick, mark=o\\
densely dashed, color=steelblue31119180, every mark/.append style={scale=2.2, solid, fill=steelblue31119180}, thick, mark=diamond\\
loosely dotted, color=steelblue31119180, every mark/.append style={scale=2, solid, fill=steelblue31119180}, thick, mark=triangle\\
densely dotted, color=steelblue31119180, every mark/.append style={scale=1.5, solid, fill=steelblue31119180}, thick, mark=square\\
densely dashdotted, color=steelblue31119180, every mark/.append style={scale=2, solid, fill=steelblue31119180}, thick, mark=pentagon\\
}

\pgfplotscreateplotcyclelist{itertol}{%
loosely dashed, color=crimson2143940, every mark/.append style={scale=1.6, solid, fill=crimson2143940}, thick, mark=*\\
densely dashed, color=darkorange25512714, every mark/.append style={scale=2.0, solid, fill=darkorange25512714}, thick, mark=diamond*\\
loosely dotted, color=steelblue31119180, every mark/.append style={scale=1.9, solid, fill=steelblue31119180}, thick, mark=triangle*\\
densely dotted, color=forestgreen4416044, every mark/.append style={scale=1.3, solid, fill=forestgreen4416044}, thick, mark=square*\\
densely dashdotted, color=orchid227119194, every mark/.append style={scale=1.9, solid, fill=orchid227119194}, thick, mark=pentagon*\\
every mark/.append style={solid, fill=gray}\\
densely dashed, every mark/.append style={solid, fill=gray}, mark=triangle*
loosely dotted, every mark/.append style={solid, fill=gray}, mark=*\\%
densely dotted, every mark/.append style={solid, fill=gray}, mark=square*\\%
densely dashdotted,every mark/.append style={solid, fill=gray}, mark=pentagon*\\%
}

\pgfplotscreateplotcyclelist{brain}{%
loosely dashed, color=crimson2143940, every mark/.append style={scale=1.6, solid, fill=crimson2143940}, thick, mark=*\\
densely dashed, color=darkorange25512714, every mark/.append style={scale=2.0, solid, fill=darkorange25512714}, thick, mark=diamond*\\
loosely dotted, color=steelblue31119180, every mark/.append style={scale=1.9, solid, fill=steelblue31119180}, thick, mark=triangle*\\
densely dotted, color=forestgreen4416044, every mark/.append style={scale=1.3, solid, fill=forestgreen4416044}, thick, mark=square*\\
every mark/.append style={solid, fill=gray}\\
}

%% file: conv_u_toy_H1.tex
\begin{tikzpicture}

  \begin{loglogaxis}[
    width=2.7in,
    height=2.7in,
    xmin=3.0e-02, xmax=8.0e-01,
    ymin=1.5e-12, ymax=5.0e-01,
    xtick={0.625, 0.3125, 0.15625, 0.078125, 0.0390625},
    xticklabels={$2^{-4}$,$2^{-5}$,$2^{-6}$,$2^{-7}$, $2^{-8}$},
    yticklabels={},
    xlabel={step size $\tau \times 10$},
    ylabel={\(\displaystyle\max_{n} \Vert u(t^{n}) - u^{n} \Vert_{\cV}\)},
    ylabel style={at={(-0.01,0.5)}},
    xmajorgrids,
    ymajorgrids,
    legend style={
      at={(0.5,-0.1)},
      anchor=north,
      /tikz/every even column/.append style={column sep=0.2cm},
    },
    legend columns=5
    legend to name=legToy, 
    cycle list name=methodcompare,
    ]
    \addplot+
    table[x={tau}, y={EPfBDF1}] {conv_u_H1.dat};
    \addplot+
    table[x={tau}, y={EPfBDF2}] {conv_u_H1.dat};
    \addplot+
    table[x={tau}, y={EPfBDF3}] {conv_u_H1.dat};
    \addplot+
    table[x={tau}, y={EPfBDF4}] {conv_u_H1.dat};
    \addplot+
    table[x={tau}, y={EPfBDF5}] {conv_u_H1.dat};
    \addplot+
    table[x={tau}, y={EPiBDF1}] {conv_u_H1.dat};
    \addplot+
    table[x={tau}, y={EPiBDF2}] {conv_u_H1.dat};
    \addplot+
    table[x={tau}, y={EPiBDF3}] {conv_u_H1.dat};
    \addplot+
    table[x={tau}, y={EPiBDF4}] {conv_u_H1.dat};
    \addplot+
    table[x={tau}, y={EPiBDF5}] {conv_u_H1.dat};
    \addplot
    table[x={tau}, y={ORDER1}] {conv_u_H1.dat};
    \addplot
    table[x={tau}, y={ORDER2}] {conv_u_H1.dat};
    \addplot
    table[x={tau}, y={ORDER3}] {conv_u_H1.dat};
    \addplot
    table[x={tau}, y={ORDER4}] {conv_u_H1.dat};
    \addplot
    table[x={tau}, y={ORDER5}] {conv_u_H1.dat};
  \end{loglogaxis}
\end{tikzpicture}
  

%% file: conv_p_toy_L2.tex
\begin{tikzpicture}

  \begin{loglogaxis}[
    width=2.7in,
    height=2.7in,
    xmin=3.0e-02, xmax=8.0e-01,
  	ymin=1.5e-12, ymax=5.0e-01,
    xtick={0.625, 0.3125, 0.15625, 0.078125, 0.0390625},
    xticklabels={$2^{-4}$,$2^{-5}$,$2^{-6}$,$2^{-7}$, $2^{-8}$},
    yticklabel style={at={(-0.7,0.5)}},
    xlabel={step size $\tau \times 10$},
  	ylabel={\(\displaystyle\max_{n} \Vert p(t^{n}) - p^{n}\Vert_{\cHQ}\)},
    ylabel style={at={(1.15,0.5)}},
    xmajorgrids,
    ymajorgrids,
    legend style={
      at={(0.5,-0.1)},
      anchor=north,
      /tikz/every even column/.append style={column sep=0.5cm},
    },
    legend columns=5,
    legend to name=legToy, 
    cycle list name=methodcompare,
    ]
  \addplot+
  table[x={tau}, y={EPfBDF1}] {conv_p_L2.dat};
  \addlegendentry{BDF1}
  \addplot+
  table[x={tau}, y={EPfBDF2}] {conv_p_L2.dat};
  \addlegendentry{BDF2}
  \addplot+
  table[x={tau}, y={EPfBDF3}] {conv_p_L2.dat};
  \addlegendentry{BDF3}
  \addplot+
  table[x={tau}, y={EPfBDF4}] {conv_p_L2.dat};
  \addlegendentry{BDF4}
  \addplot+
  table[x={tau}, y={EPfBDF5}] {conv_p_L2.dat};
  \addlegendentry{BDF5}
  \addplot+
  table[x={tau}, y={EPiBDF1}] {conv_p_L2.dat};
  \addlegendentry{BDF1}
  \addplot+
  table[x={tau}, y={EPiBDF2}] {conv_p_L2.dat};
  \addlegendentry{BDF2}
  \addplot+
  table[x={tau}, y={EPiBDF3}] {conv_p_L2.dat};
  \addlegendentry{BDF3}
  \addplot+
  table[x={tau}, y={EPiBDF4}] {conv_p_L2.dat};
  \addlegendentry{BDF4}
  \addplot+
  table[x={tau}, y={EPiBDF5}] {conv_p_L2.dat};
  \addlegendentry{BDF5}
  \addplot
  table[x={tau}, y={ORDER1}] {conv_p_L2.dat};
  \addlegendentry{order 1}
  \addplot
  table[x={tau}, y={ORDER2}] {conv_p_L2.dat};
  \addlegendentry{order 2}
  \addplot
  table[x={tau}, y={ORDER3}] {conv_p_L2.dat};
  \addlegendentry{order 3}
  \addplot
  table[x={tau}, y={ORDER4}] {conv_p_L2.dat};
  \addlegendentry{order 4}
  \addplot
  table[x={tau}, y={ORDER5}] {conv_p_L2.dat};
  \addlegendentry{order 5}
\end{loglogaxis}
\end{tikzpicture}

%% file: itertol_error_BDF1.tex
\begin{tikzpicture}

  \begin{loglogaxis}[
    width=2in,
    height=2in,
    xmin=1.2e-03, xmax=8.0e-02,
    ymin=1.5e-4, ymax=8.0e-02,
    xtick={0.0625, 0.015625, 0.00390625}, 
    xticklabels={$2^{-4}$,$2^{-6}$, $2^{-8}$},
    xlabel={step size $\tau$},
    xmajorgrids,
    ymajorgrids,
    legend style={
      at={(0.5,-0.1)},
      anchor=north,
      /tikz/every even column/.append style={column sep=0.3cm},
    },
    cycle list name=itertol,
    ]
  \addplot
  table[x={tau}, y={EXP1.00e+00}] {itertol_BDF1_error.dat};
  \addplot
  table[x={tau}, y={EXP1.50e+00}] {itertol_BDF1_error.dat};
  \addplot
  table[x={tau}, y={EXP2.00e+00}] {itertol_BDF1_error.dat};
  \addplot
  table[x={tau}, y={EXP2.50e+00}] {itertol_BDF1_error.dat};
 \addplot
 table[x={tau}, y={EXP3.00e+00}] {itertol_BDF1_error.dat};
  \addplot
  table[x={tau}, y={ORDER}] {itertol_BDF1_error.dat};
  \addplot
  table[x={tau}, y={EPi}] {itertol_BDF1_error.dat};
\end{loglogaxis}
\end{tikzpicture}

%% file: itertol_error_BDF2.tex
\begin{tikzpicture}

  \begin{loglogaxis}[
    width=2in,
    height=2in,
    xmin=1.2e-03, xmax=8.0e-02,
    ymin=3.e-8, ymax=5.0e-03,
    xtick={0.0625, 0.015625, 0.00390625}, 
    xticklabels={$2^{-4}$,$2^{-6}$, $2^{-8}$},
    xlabel={step size $\tau$},
    xmajorgrids,
    ymajorgrids,
    cycle list name=itertol,
    ]
  \addplot
  table[x={tau}, y={EXP2.00e+00}] {itertol_BDF2_error.dat};
  \addplot
  table[x={tau}, y={EXP2.50e+00}] {itertol_BDF2_error.dat};
  \addplot
  table[x={tau}, y={EXP3.00e+00}] {itertol_BDF2_error.dat};
  \addplot
  table[x={tau}, y={EXP3.50e+00}] {itertol_BDF2_error.dat};
  \addplot
  table[x={tau}, y={EXP4.00e+00}] {itertol_BDF2_error.dat};
  \addplot
  table[x={tau}, y={ORDER}] {itertol_BDF2_error.dat};
  \addplot
  table[x={tau}, y={EPi}] {itertol_BDF2_error.dat};
\end{loglogaxis}
\end{tikzpicture}

%% file: itertol_error_BDF3.tex
\begin{tikzpicture}

  \begin{loglogaxis}[
    width=2in,
    height=2in,
    xmin=1.2e-03, xmax=8.0e-02,
    ymin=1.1e-10, ymax=9.0e-04,
    xtick={0.0625, 0.015625, 0.00390625}, 
    xticklabels={$2^{-4}$,$2^{-6}$, $2^{-8}$},
    xlabel={step size $\tau$},
    ylabel={error},
    ylabel style={at={(1.2,0.49)}},
    xmajorgrids,
    ymajorgrids,
    legend style={
      at={(0.5,-0.1)},
      anchor=north,
      /tikz/every even column/.append style={column sep=0.3cm},
    },
    legend columns=3,
    legend to name=legITERTOLC, 
    cycle list name=itertol,
    ]
  \addplot
  table[x={tau}, y={EXP3.00e+00}] {itertol_BDF3_error.dat};
  \addlegendentry{$\tol\cong\tau^{k}$}
  \addplot
  table[x={tau}, y={EXP3.50e+00}] {itertol_BDF3_error.dat};
  \addlegendentry{$\tol\cong\tau^{k+1/2}$}
  \addplot
  table[x={tau}, y={EXP4.00e+00}] {itertol_BDF3_error.dat};
  \addlegendentry{$\tol\cong\tau^{k+1}$}
  \addplot
  table[x={tau}, y={EXP4.50e+00}] {itertol_BDF3_error.dat};
  \addlegendentry{$\tol\cong\tau^{k+3/2}$}
 \addplot
 table[x={tau}, y={EXP5.00e+00}] {itertol_BDF3_error.dat};
  \addlegendentry{$\tol\cong\tau^{k+2}$}
  \addplot
  table[x={tau}, y={ORDER}] {itertol_BDF3_error.dat};
  \addlegendentry{order (1, 2 or 3)}
  \addplot
  table[x={tau}, y={EPi}] {itertol_BDF3_error.dat};
  \addlegendentry{implicit error}
\end{loglogaxis}
\end{tikzpicture}

%% file: brain_conv.tex
\begin{tikzpicture}

    \begin{loglogaxis}[
      width=2.35in,
      height=2.35in,
      xmin=2.8e-03, xmax=8.0e-02,
      ymin=1.0e-4, ymax=0.6,
      xtick={0.0625, 0.03125, 0.015625, 0.0078125, 0.00390625}, 
      xticklabels={$2^{-4}$,$2^{-5}$,$2^{-6}$,$2^{-7}$, $2^{-8}$,}, 
      xlabel={step size $\tau$},
      ylabel={error},
      ytick pos=left,
      xmajorgrids,
      ymajorgrids,
      legend pos=outer north east,
      legend cell align=left,
      cycle list name=itertol,
      ]
    \addplot
    table[x={tau}, y={NetworkEPfBDF2}] {brain_p1_L2.dat};
    \addlegendentry{$\|p_1\|_{\cHQ}$}
    \addplot
    table[x={tau}, y={NetworkEPfBDF2}] {brain_p2_L2.dat};
    \addlegendentry{$\|p_2\|_{\cHQ}$}
    \addplot
    table[x={tau}, y={NetworkEPfBDF2}] {brain_p3_L2.dat};
    \addlegendentry{$\|p_3\|_{\cHQ}$}
    \addplot
    table[x={tau}, y={NetworkEPfBDF2}] {brain_u_H1.dat};
    \addlegendentry{$\|u\|_{\cV}$}
    \addplot
    table[x={tau}, y={ORDER2}] {brain_u_H1.dat};
    \addlegendentry{order\,2}
  \end{loglogaxis}
  \end{tikzpicture}
  
  